\documentclass[12pt]{amsart}
\pdfoutput=1

\usepackage{etex}
\usepackage{amsmath, amssymb}
\usepackage{array}
\usepackage[frame,cmtip,arrow,matrix,line,graph,curve]{xy}
\usepackage{graphpap, color, paralist, pstricks}
\usepackage[mathscr]{eucal}
\usepackage[pdftex]{graphicx}
\usepackage[pdftex,colorlinks,backref=page,citecolor=blue]{hyperref}
\usepackage{tikz}
\usepackage{subcaption}
\usepackage{rotating}

\numberwithin{equation}{section}
\numberwithin{figure}{section}

\newtheorem{theorem}{Theorem}[section]

\newtheorem{proposition}[theorem]{Proposition}

\newtheorem{lemma}[theorem]{Lemma}

\theoremstyle{definition}
\newtheorem{definition}[theorem]{Definition}
\newtheorem{example}[theorem]{Example}

\newtheorem{remark}[theorem]{Remark}

\newcommand{\CC}{\mathbb{C} }

\newcommand{\ZZ}{\mathbb{Z} }

\newcommand{\RR}{\mathbb{R} }
\newcommand{\NN}{\mathbb{N} }

\newcommand{\cA}{\mathcal{A} }
\newcommand{\cC}{\mathcal{C} }

\newcommand{\cS}{\mathcal{S} }

\def\spec{\mathrm{Spec} }

\hypersetup{%
pdftitle={Presentations of Roger-Yang generalized skein algebra},%
pdfauthor={Farhan Azad, Zixi Chen, Matt Dreyer, Ryan Horowitz, Han-Bom Moon},%
pdfkeywords={skein algebra, presentation},%
citecolor=blue,%
linkcolor=blue,%
}

\begin{document}

\title{Presentations of the Roger-Yang generalized skein algebra}
\date{\today}

\author{Farhan Azad}
\address{Department of Mathematics, Fordham University, New York, NY 10023}
\email{fazad2@fordham.edu}

\author{Zixi Chen}
\address{Department of Mathematics, Fordham University, New York, NY 10023}
\email{zchen218@fordham.edu}

\author{Matt Dreyer}
\address{Department of Mathematics, Cornell University, New York, NY 14853}
\email{mjd367@cornell.edu}

\author{Ryan Horowitz}
\address{Department of Mathematics, New York University, New York, NY 10012}
\email{reh420@nyu.edu}

\author{Han-Bom Moon}
\address{Department of Mathematics, Fordham University, New York, NY 10023}
\email{hmoon8@fordham.edu}

\begin{abstract}
We describe presentations of the Roger-Yang generalized skein algebras for punctured spheres with an arbitrary number of punctures. This skein algebra is a quantization of the decorated Teichm\"uller space and generalizes the construction of the Kauffman bracket skein algebra. In this paper, we also obtain a new interpretation of the homogeneous coordinate ring of the Grassmannian of planes in terms of skein theory. 
\end{abstract}

\maketitle

\section{Introduction}\label{sec:intro}

Since the Kauffman bracket skein algebra $\cS_{q}(\Sigma)$ of a closed surface $\Sigma$ was introduced by Przytycki (\cite{Przytycki91}) and Turaev (\cite{Turaev88}), based on Kauffman's skein theoretic description of the Jones polynomial (\cite{Kauffman87}), it has been one of the central objects in low-dimensional quantum topology. It has interesting connections with many branches of mathematics, including character varieties (\cite{Bullock97, BullockFrohmanJKB99, PrzytyckiSikora00}), Teichm\"uller spaces and hyperbolic geometry (\cite{BonahonWong11}), and cluster algebras (\cite{FominShapiroThurston08, Muller16}). 

Roger and Yang extended skein algebras to oriented surfaces with \emph{punctures} and defined the algebra $\cA_{q}(\Sigma)$ (\cite{RogerYang14}) by including arc classes. The algebra $\cA_{q}(\Sigma)$ is indeed a quantization of the decorated Teichm\"uller space (\cite{Penner87, RogerYang14}) and is also compatible with the cluster algebra from surfaces (\cite{MoonWong20}). Thus, it can be regarded as a good extension of $\cS_{q}(\Sigma)$ and strengthens the connections of the aforementioned subjects. 

For both $\cS_{q}(\Sigma)$ and $\cA_{q}(\Sigma)$, many algebraic properties have been shown. For example, they are finitely generated algebras (\cite{Bullock99, BKPW16JKTR}) without zero divisors (\cite{PrzytyckiSikora00, BonahonWong11, MoonWong19, MoonWong20}) with a few exceptions. However, very few examples of $\cS_{q}(\Sigma)$ and $\cA_{q}(\Sigma)$ with explicit presentations are known. If we denote by $\Sigma_{g, n}$ (resp. $\Sigma_{g}^{k}$) the oriented surface with genus $g$ and $n$ punctures (resp. $k$ boundary components), then a presentation of $\cS_{q}(\Sigma_{g}^{k})$ is known only for $g = 0$, $k \le 4$ and $g = 1$, $k \le 2$ cases (\cite{BullockPrzytycki00}). The presentation of $\cA_{q}(\Sigma_{g, n})$ is known for $g = 0$, $n \le 3$ and $g = 1$, $n \le 1$ (\cite{BKPW16Involve}).

The main result of this paper is a calculation of a presentation of $\cA_{q}(\Sigma_{0,n})$ for arbitrary $n$. Arrange $n$ punctures $v_{1}, v_{2}, \cdots, v_{n}$ in a small circle $C$ on $S^{2}$ clockwise. Let $\beta_{ij} = \beta_{ji}$ be the geodesic in $C$, which connects $v_{i}$ and $v_{j}$. 

\begin{theorem}[Theorem \ref{thm:matinthm}]\label{thm:mainthmintro}
The algebra $\cA_{q}(\Sigma_{0,n})$ is isomorphic to 
\[
	\ZZ[q^{\pm \frac{1}{2}}, v_{1}^{\pm}, v_{2}^{\pm}, \cdots, v_{n}^{\pm}]\langle \beta_{ij}\rangle_{1\le i < j \le n}/J
\]
where $J$ is the ideal generated by 
\begin{enumerate}
\item (Ptolemy relations) For any 4-subset $I = \{i, j, k, \ell\} \subset [n]$ in cyclic order, $\beta_{ik}\beta_{j\ell} = q\beta_{i\ell}\beta_{jk}+q^{-1}\beta_{ij}\beta_{k\ell}$;
\item (Quantum commutation relations) For any 4-subset $I = \{i, j, k, \ell\} \subset [n]$ in cyclic order, $\beta_{ij}\beta_{k\ell} = \beta_{k\ell}\beta_{ij}$. For any $3$-subset $I = \{i, j, k\} \subset [n]$ in cyclic order, $\beta_{jk}\beta_{ij} = q\beta_{ij}\beta_{jk} + (q^{-\frac{1}{2}}-q^{\frac{3}{2}})v_{j}^{-1}\beta_{ik}$;
\item ($\gamma$-relations) For any $i, j \in [n]$, $\gamma_{ij}^{+} = \gamma_{ij}^{-}$;
\item (Big circle relation) $\delta = -q^{2} - q^{-2}$.
\end{enumerate}
\end{theorem}

The definition of $\gamma_{ij}^{\pm}$ and $\delta$, as well as their formulas in terms of the $\beta_{ij}$'s, are given in Section \ref{sec:relations}. We want to emphasize that each generator of $J$ has a very simple and explicit topological interpretation. See Section \ref{sec:relations} for the details. 

A key step of the proof is the computation of a presentation of $\cA_{q}(\RR^{2}_{n})$ (Section \ref{sec:presentationR2}), where $\RR^{2}_{n}$ is the plane with $n$ punctures. By finding a generating set and many relations (Sections \ref{sec:genetors} and \ref{sec:relations}), it is straightforward to construct a surjective homomorphism of the form
\[
	\bar{f} : \ZZ[q^{\pm \frac{1}{2}}, v_{1}^{\pm}, v_{2}^{\pm}, \cdots, v_{n}^{\pm}]\langle \beta_{ij}\rangle/K \to \cA_{q}(\RR^{2}_{n}),
\]
where $K$ is the ideal generated by Ptolemy relations and Quantum commutation relations. 

Similar to many other problems of finding presentations, a difficult non-trivial step is to show the injectivity of $\bar{f}$. To do so, we employ a technique from algebraic geometry, in particular the dimension theory. When $q = 1$, $\bar{f}$ is a surjective homomorphism of commutative algebras. The affine variety associated to $\CC \otimes_{\ZZ}\cA_{q}(\RR^{2}_{n})$ is a closed subvariety of the affine variety associated to $\CC \otimes_{\ZZ}\ZZ[q^{\pm \frac{1}{2}}, v_{1}^{\pm}, v_{2}^{\pm}, \cdots, v_{n}^{\pm}]\langle \beta_{ij}\rangle/K$. They have the same dimension and the latter is irreducible. Therefore, they are isomorphic and $\bar{f}$ is an isomorphism.

\begin{remark}\label{rem:ringSintro}
During the proof, we show that the presentation of $\cA_{q}(\RR^{2}_{n})$ with $q = 1$ is a ring extension of the homogeneous coordinate ring of the Grassmannian of planes. The ring has occurred in many different territories of mathematics including classical invariant theory, cluster algebras, and even computational biology (Remarks \ref{rem:ringS}, \ref{rem:interpretationofS}). Our result provides a skein theoretic interpretation of the same object.
\end{remark}

\begin{remark}\label{rem:smallnintro}
The method of the proof relies on the fact that $\cA_{q}(\Sigma_{0, n+1})$ is a domain, which was shown in \cite{MoonWong19} for $n \ge 3$. Thus, the proof is valid for $n \ge 3$. However, even for $n \le 2$, our presentation still coincides with the calculation in \cite{BKPW16Involve}. See Remark \ref{rem:smalln}. 
\end{remark}

\subsection*{Acknowledgements}
The last author thanks Helen Wong for helpful discussions and many valuable suggestions. The authors also thank the anonymous referees for valuable comments on earlier drafts of this paper.


\section{The Roger-Yang generalized skein algebra}\label{sec:skeinalgebra}

In this section, we present the definition and basic properties of the Roger-Yang generalized skein algebra $\cA_{q}(\Sigma)$.

Let $\overline{\Sigma}$ be an oriented surface without boundary, not necessarily compact nor connected. Let $V \subset \overline{\Sigma}$ be a finite subset of points and let $\Sigma = (\overline{\Sigma}, V)$. A point $v \in V$ is called a \emph{puncture} and $\Sigma$ is called a \emph{punctured surface}. We allow the case that $V = \emptyset$. In this paper, there are two relevant examples of a punctured surface. Let $\Sigma_{g, n}$ be the $n$-punctured genus $g$ surface. Let $\RR^{2}_{n}$ be the $n$-punctured plane. If $V$ is any $n$-subset of $\RR^{2}$, then $\RR^{2}_{n} = (\RR^{2}, V)$. 

\begin{definition}\label{def:curve}
Fix a punctured surface $\Sigma = (\overline{\Sigma} , V)$. A multicurve is a one-dimensional compact submanifold $\Gamma$ (possibly with boundary) of $\overline{\Sigma} \times (0, 1)$ satisfying the following properties:
\begin{enumerate}
\item $\partial \Gamma = V \times (0, 1) \cap \Gamma$;
\item the composition map $\Gamma \to \overline{\Sigma} \times (0, 1) \to \overline{\Sigma}$ is a generic immersion. 
\end{enumerate}
A \emph{curve} is a connected multicurve. A \emph{loop} is a curve without boundary, and an \emph{arc} is a curve with boundary. 
\end{definition}

To visualize a curve, we draw its \emph{diagram}. The second coordinate $t \in (0, 1)$ is the vertical coordinate oriented toward the reader. It encodes which strand is over/under another strand, as in Figure \ref{fig:Example1}. 

\begin{figure}[!ht]
\begin{minipage}{0.7in}\includegraphics[width=\textwidth]{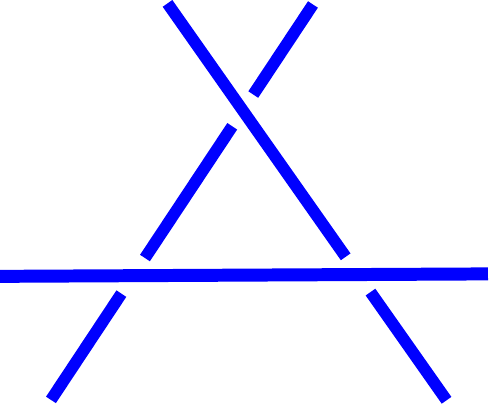}\end{minipage} \quad \quad
\begin{minipage}{0.7in}\includegraphics[width=\textwidth]{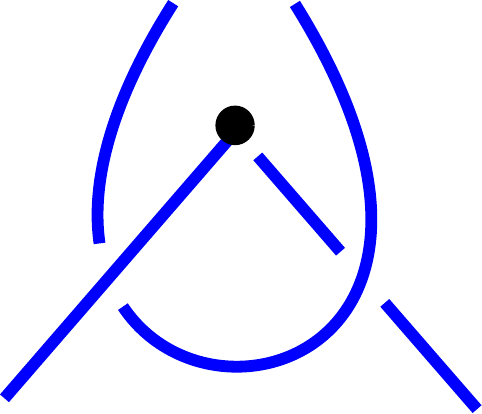}\end{minipage}
\caption{Examples of local planar diagram for curves}
\label{fig:Example1}
\end{figure}

We will always think about the \emph{regular isotopy classes} of multicurves. Roughly, two multicurves are regular isotopic if (1) they are homotopic, (2) each step in the deformation is a multicurve in the above sense, and (3) the deformation does not involve a Reidemeister move of type I. For the precise definition, consult \cite[Section 2]{RogerYang14}. We may assume that for any multicurve, the only multiple points on $\Sigma$ in the planar diagram above are double points. However, note that it is possible that there are more than two strands meeting at a puncture. 

There is a natural \emph{stacking operation} of multicurves. Let $\alpha, \beta$ be two multicurves. By rescaling the vertical coordinate, we may assume that $\alpha \subset \overline{\Sigma} \times (0, \frac{1}{2})$ and $\beta \subset \overline{\Sigma} \times (\frac{1}{2}, 1)$. Then $\alpha * \beta$ is defined as `stacking' $\beta$ over $\alpha$: $\alpha *\beta := \alpha \cup \beta$. 

\begin{definition}\label{def:skeinalgebra}
Let $\Sigma = (\overline{\Sigma}, V)$ be a punctured surface. Suppose that $V = \{v_{1}, v_{2}, \cdots, v_{n}\}$. Let $R_{q, n} := \ZZ[q^{\pm \frac{1}{2}}, v_{1}^{\pm}, v_{2}^{\pm}, \cdots, v_{n}^{\pm}]$, which is the commutative Laurent polynomial ring with respect to $q^{\frac{1}{2}}, v_{1}, \cdots, v_{n}$ with integer coefficients. The \emph{generalized skein algebra} $\cA_{q}(\Sigma)$ is an $R_{q, n}$-algebra generated by regular isotopy classes of multicurves in $\Sigma$. The addition and scalar multiplication are formal, but the multiplication is given by the stacking operation $\alpha \beta := \alpha *\beta$. The algebra $\cA_{q}(\Sigma)$ has four types of relations:

\begin{tabular}{llrcl}
1) &\mbox{Skein relation} & \quad \begin{minipage}{0.4in}\includegraphics[width=\textwidth]{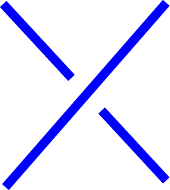}\end{minipage} & $=$ & $q$ \begin{minipage}{0.4in}\includegraphics[width=\textwidth]{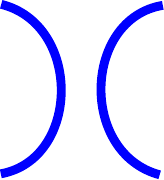}\end{minipage} $+\; q^{-1}$ \begin{minipage}{0.4in}\includegraphics[width=\textwidth]{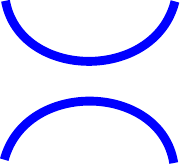}\end{minipage}\\
2) & \mbox{Puncture-Skein relation} & \begin{minipage}{0.4in}\includegraphics[width=\textwidth]{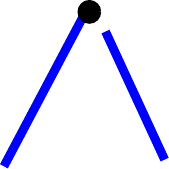}\end{minipage} & $=$ & $v^{-1}\left(q^{\frac{1}{2}}\right.$ \begin{minipage}{0.4in}\includegraphics[width=\textwidth]{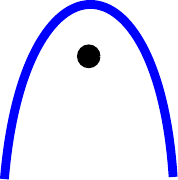}\end{minipage} $+\; q^{-\frac{1}{2}}$ \begin{minipage}{0.4in}\includegraphics[width=\textwidth]{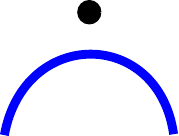}\end{minipage}$\left.\right)$\\
3) & \mbox{Framing relation} & \quad \begin{minipage}{0.35in}\includegraphics[width=\textwidth]{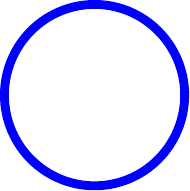}\end{minipage} & $=$ & $-q^{2} - q^{-2}$\\[10pt]
4) & \mbox{Puncture-Framing relation} & \quad \begin{minipage}{0.35in}\includegraphics[width=\textwidth]{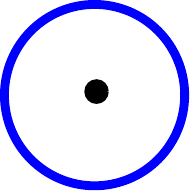}\end{minipage} & $=$ & $q+q^{-1}$
\end{tabular}
\end{definition}

\begin{example}\label{ex:skeinexample}
For each vertex $v_{i} \in V$, the \emph{waterdrop} $\omega_{i}$ at $v_{i}$ is the small arc class starting at $v_{i}$, turning around counterclockwise, and ending at $v_{i}$. We assume that the ending point is higher than the starting point. By using the Puncture-Skein relation, Framing relation, and Puncture-Framing relation, one may check that 
\[
	\omega_{i} = \mbox{\begin{minipage}{0.2in}\includegraphics[width=\textwidth]{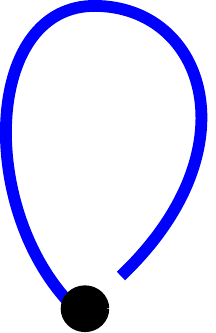}\end{minipage}} = v_{i}^{-1}\left(q^{\frac{1}{2}} \mbox{\begin{minipage}{0.25in}\includegraphics[width=\textwidth]{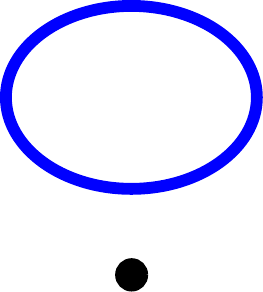}\end{minipage}} + q^{-\frac{1}{2}}\mbox{\begin{minipage}{0.25in}\includegraphics[width=\textwidth]{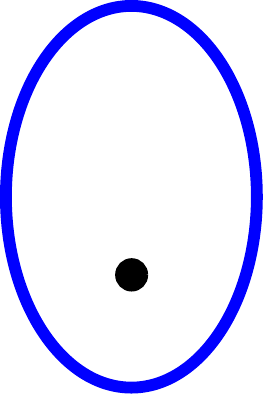}\end{minipage}}\right) = \left(q^{\frac{1}{2}}(-q^{2} - q^{-2}) + q^{-\frac{1}{2}}(q+q^{-1})\right)v_{i}^{-1} = (q^{\frac{1}{2}} - q^{\frac{5}{2}})v_{i}^{-1}.
\]
\end{example}

\begin{example}\label{ex:conjugate}
For any curve class $\alpha \in \cA_{q}(\Sigma)$, the \emph{conjugate} $\overline{\alpha}$ of $\alpha$ is the new curve obtained by reversing all of the crossing data. Equivalently, it is induced by the map $(x, t) \mapsto (x, 1-t)$ from $\Sigma \times (0, 1)$ to itself. Then the conjugation map $\alpha \mapsto \overline{\alpha}$ is an anti-involution on $\cA_{q}(\Sigma)$. A calculation shows that $\overline{\omega_{i}} = (q^{-\frac{1}{2}} - q^{-\frac{5}{2}})v_{i}^{-1} = -q^{3}\omega_{i} = q\omega_{i} + (q^{-\frac{1}{2}} - q^{\frac{3}{2}})(-q^{2} - q^{-2})v_{i}^{-1}$. 
\end{example}

\begin{remark}
The original definition in \cite{RogerYang14} is a ring-theoretic completion of $\cA_{q}(\Sigma)$ in Definition \ref{def:skeinalgebra}. The original paper \cite{RogerYang14} did not address non-compact $\overline{\Sigma}$ cases, but the construction can be done in the same way. However, its connection to hyperbolic geometry (\cite[Section 3]{RogerYang14}) cannot be directly extended.  
\end{remark}

In \cite{RogerYang14}, the authors defined the \emph{curve algebra} $\cC(\Sigma)$ which is the classical limit of $\cA_{q}(\Sigma)$, which can be described by using \emph{immersed} curves on $\overline{\Sigma}$. For the detail of the construction, see \cite[Section 2.2]{RogerYang14}. Algebraically, the curve algebra $\cC(\Sigma)$ can be obtained by setting $q^{\frac{1}{2}} = 1$, i.e., $\cC(\Sigma) = \cA_{1}(\Sigma) = \cA_{q}(\Sigma)/(q^{\frac{1}{2}} - 1)$. Thus, if we set 
\begin{equation}\label{eqn:Rn}
	R_{n} := R_{q, n}/(q^{\frac{1}{2}}-1) \cong \ZZ[v_{1}^{\pm},  v_{2}^{\pm}, \cdots, v_{n}^{\pm}],
\end{equation}
$\cC(\Sigma)$ is an $R_{n}$-algebra. The algebra $\cC(\Sigma)$ is a commutative algebra and has a Poisson algebra structure. Moreover, $\cA_{q}(\Sigma)$ is a deformation quantization of $\cC(\Sigma)$ (\cite[Theorem 2.13]{RogerYang14}). 

We leave a few known structural results on $\cA_{q}(\Sigma)$. 

\begin{theorem}[\protect{\cite[Theorem 1]{Bullock99}, \cite[Theorem 2.2]{BKPW16JKTR}}]\label{thm:finitegeneration}
The algebra $\cA_{q}(\Sigma_{g, n})$ (and hence $\cC(\Sigma_{g, n})$) is finitely generated. 
\end{theorem}

The proof of \cite[Theorem 10.5]{MoonWong19} tells us the following result. 

\begin{theorem}[\protect{\cite[Theorem 10.5]{MoonWong19}}]\label{thm:classicalimpliesquantum}
Let $\Sigma$ be a punctured surface. If $\cC(\Sigma)$ is an integral domain, then $\cA_{q}(\Sigma)$ is a domain (there is no zero divisor). 
\end{theorem}

\begin{theorem}[\protect{\cite[Theorem 5.1 and Section 4]{MoonWong19}}]\label{thm:domain}
There is a function $f(g)$ such that for $n \ge f(g)$, $\cC(\Sigma_{g, n})$ (and hence $\cA_{q}(\Sigma_{g, n})$ by Theorem \ref{thm:classicalimpliesquantum}) is a domain. When $g = 0$, $\cA_{q}(\Sigma_{0, n})$ and $\cC(\Sigma_{0, n})$ are domains for $n \ge 4$. 
\end{theorem}

\begin{remark}\label{rem:smallncases}
When $n = 0$, $\cA_{q}(\Sigma_{g, 0})$ is the classical Kauffman skein algebra $\cS_{q}(\Sigma_{g, 0})$. Przytycki and Sikora showed that $\cA_{q}(\Sigma_{g, 0})$ is a domain (\cite{PrzytyckiSikora00}).
\end{remark}

For any finitely generated commutative algebra $A$ over $\CC$ we may define the (Krull) dimension of $A$ (\cite[Section I.1]{Hartshorne77}). This is equal to the dimension of its associated affine algebraic variety $\spec \; A$ (\cite[Proposition I.1.7.]{Hartshorne77}). When $A$ is an integral domain, this is equal to the transcendental degree of the field of fractions $Q(A)$ of $A$ (\cite[Exercise II.3.20.]{Hartshorne77}). 

\begin{proposition}\label{prop:dimension}
Let $T$ be a triangulation of $\Sigma_{g, n}$ where its zero-skeleton is the set of punctures of $\Sigma_{g, n}$. For the same range of $n$ in Theorem \ref{thm:domain}, the field of fractions $Q(\CC \otimes_{\ZZ} \cC(\Sigma_{g, n}))$ is transcendentally generated by edges in $T$. Thus, the dimension of the $\CC$-algebra $\CC \otimes_{\ZZ} \cC(\Sigma_{g, n})$ is the number of edges for a triangulation $T$, which is $6g - 6 + 3n$. In particular, for $n \ge 4$, $\dim \CC \otimes_{\ZZ} \cC(\Sigma_{0, n}) = 3n - 6$. 
\end{proposition}

\begin{proof}
By \cite[Lemma 3.4]{MoonWong19}, after a certain localization, $\CC \otimes_{\ZZ}\cC(\Sigma_{g, n})$ is isomorphic to $\CC [\lambda_{i}^{\pm}]$, where $\lambda_{i}$ is a variable for each edge $x_{i}$ in $T$. The localization does not affect  the field of fractions, so $Q(\CC \otimes_{\ZZ} \cC(\Sigma_{g, n})) \cong Q(\CC[\lambda_{i}^{\pm}]) \cong \CC(\lambda_{i})$. The remaining statements are immediate from an Euler characteristic calculation.
\end{proof}

Finally, we state the following module theoretic result. 

\begin{definition}\label{def:reducedcurves}
A \emph{reduced} curve is a curve class without any self-crossing (both on the interior and at a puncture) on the planar diagram that is neither a trivial loop nor a punctured loop. A multicurve is \emph{reduced} if it is a finite union of reduced curves without any crossings. For a notational convention, we will regard the empty set as a reduced multicurve.
\end{definition}

\begin{proposition}\label{prop:freeness}
Fix a surface $\Sigma = (\overline{\Sigma}, V)$ with $|V| = n$. The algebra $\cA_{q}(\Sigma)$ (resp. $\cC(\Sigma)$) is a free $R_{q, n}$-module (resp. $R_{n}$-module) with a basis consisting of reduced multicurves, with one exception when $\Sigma = \Sigma_{0, 1}$ (see Remark \ref{rem:smalln}). 
\end{proposition}

\begin{proof}
The classical case is in the proof of \cite[Theorem 2.4]{RogerYang14}. The quantum case is obtained from the fact that $\cA_{q}(\Sigma)$ is embedded into a topologically free algebra (\cite[Theorem 2.4]{RogerYang14}) which does not have any relation among basis vectors. 
\end{proof}


\section{Generators}\label{sec:genetors}

In this section, we describe a collection of curves in $\RR^{2}_{n}$ and $\Sigma_{0,n}$ and show that they generate $\cA_{q}(\RR^{2}_{n})$ and $\cA_{q}(\Sigma_{0,n})$ as $R_{q, n}$-algebras.  

Let $\overline{\Sigma}$ be $\RR^{2}$ or $S^{2}$. We may arrange $n$ punctures arbitrarily. Take a small circle $C$ on $\overline{\Sigma}$. From now on, we assume that the $n$ punctures $v_{1}, v_{2}, \cdots, v_{n}$ lie on $C$ in clockwise order. Let $P$ be the convex polygon inscribed in $C$, whose vertices are $v_{1}, v_{2}, \cdots, v_{n}$ (Figure \ref{fig:PDR}).

\begin{definition}\label{def:betacurves}
For any pair $i < j$ in $[n] := \{1, 2, \cdots, n\}$, let $\beta_{ij}$ be the regular isotopy class of the geodesic in $C$, which connects $v_{i}$ and $v_{j}$. For notational convenience, we set $\beta_{ji} = \beta_{ij}$. 
\end{definition}

\begin{figure}[!ht]
\begin{minipage}{2in}\includegraphics[width=\textwidth]{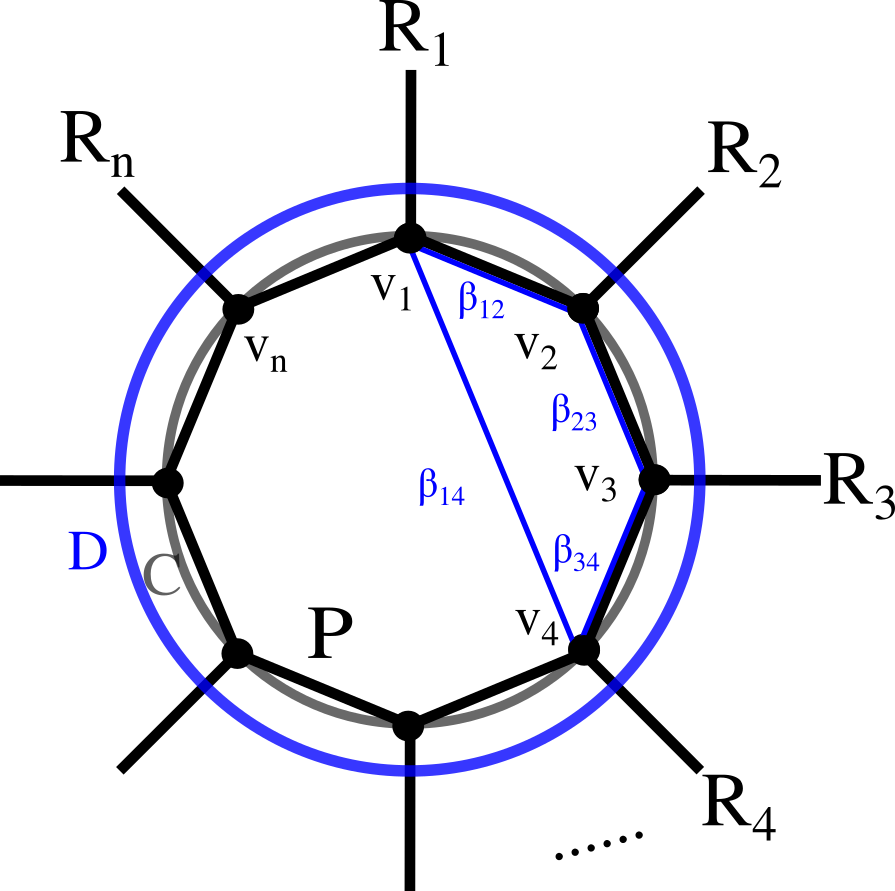}\end{minipage}
\caption{$\beta_{ij}$ classes and configuration of $P$, $D$, and $R$.}
\label{fig:PDR}
\end{figure}

Thus, $\beta_{ij}$ classes can be represented by the diagonals and sides of $P$.

\begin{proposition}\label{prop:generators}
As an $R_{q, n}$-algebra, $\cA_{q}(\Sigma_{0,n})$ and $\cA_{q}(\RR^{2}_{n})$ are generated by $\{\beta_{ij}\}_{1 \le i < j \le n}$. 
\end{proposition}

\begin{proof}
Any multicurve $\alpha$ is generated by reduced multicurves by Proposition \ref{prop:freeness}. Each multicurve is a product of reduced curves. Thus, $\cA_{q}(\Sigma_{0, n})$ and $\cA_{q}(\RR^{2}_{n})$ are generated by reduced curves. 

By \cite[Proposition 2.2]{BKPW16Involve}, any reduced loop class is generated by reduced arc classes. For the reader's convenience, we describe an example of the recursive relation in Figure \ref{fig:recursion}. Indeed, in \cite[Proposition 2.2]{BKPW16Involve}, the authors proved the statement for $\Sigma_{0,n}$ only. However, their proof only relies on the fact that every loop in $\Sigma_{0,n}$ divides the surface into two components. Thus, the same proof works for $\RR^{2}_{n}$.

\begin{figure}[!ht]
\begin{minipage}{0.7in}\includegraphics[width=\textwidth]{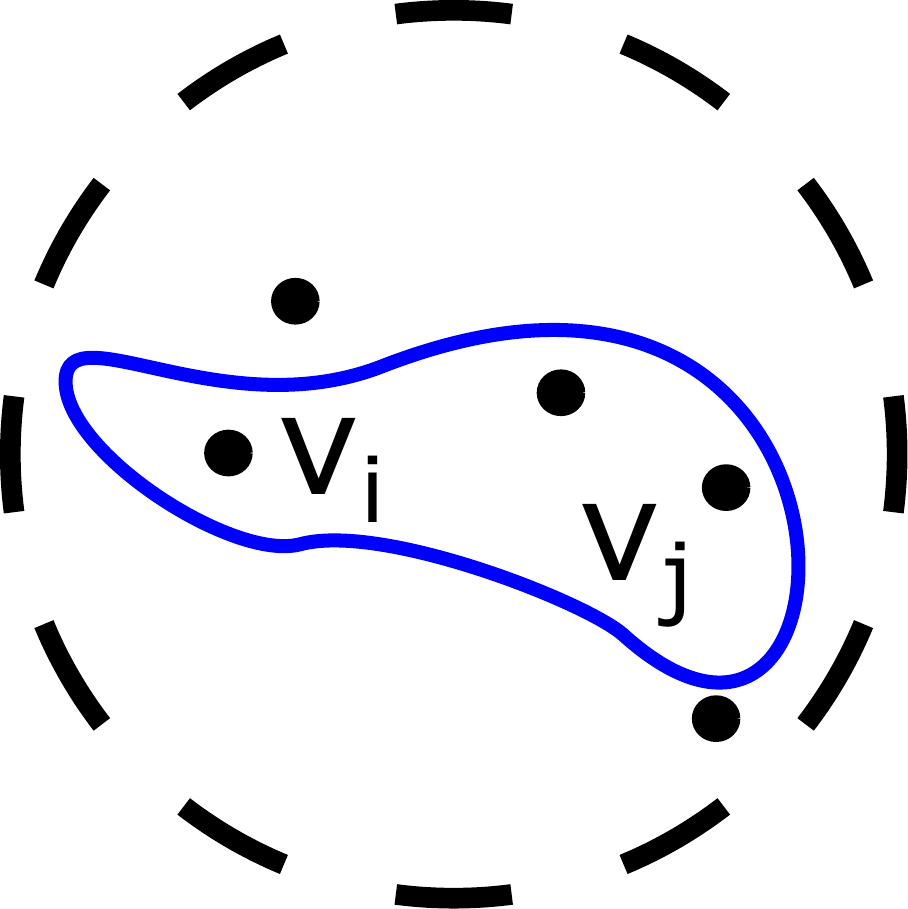}\end{minipage} $\;= \;v_{i}v_{j}$ \begin{minipage}{0.7in}\includegraphics[width=\textwidth]{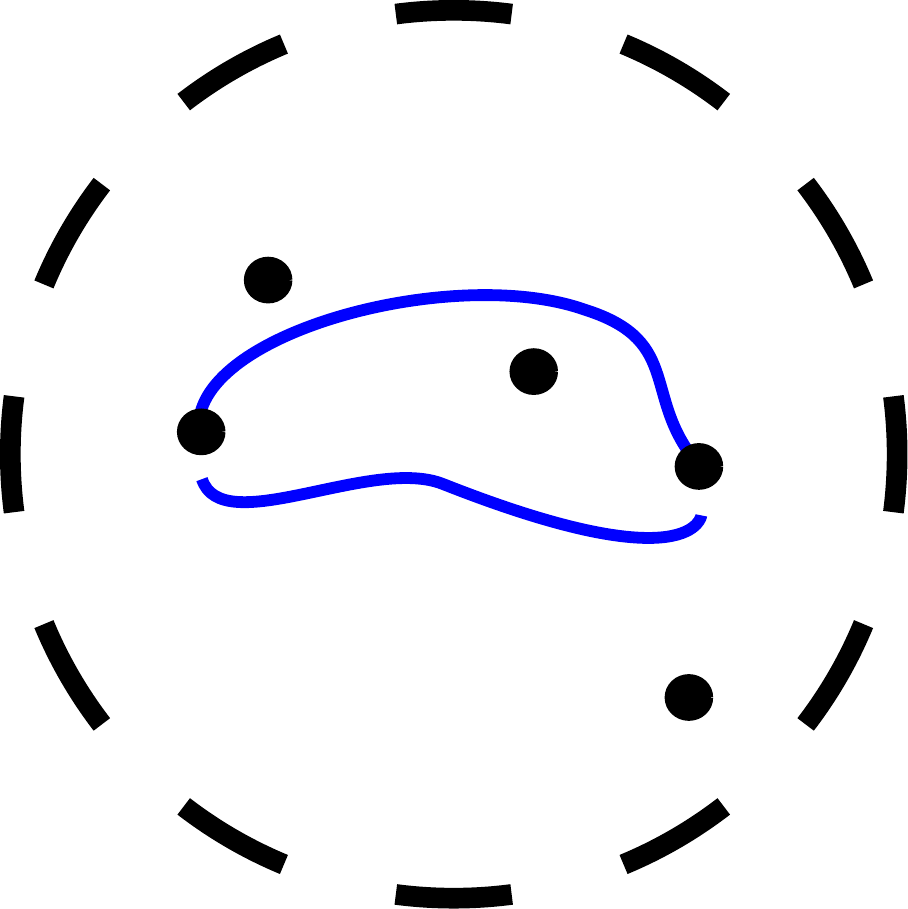}\end{minipage} $\;-\; q$ \begin{minipage}{0.7in}\includegraphics[width=\textwidth]{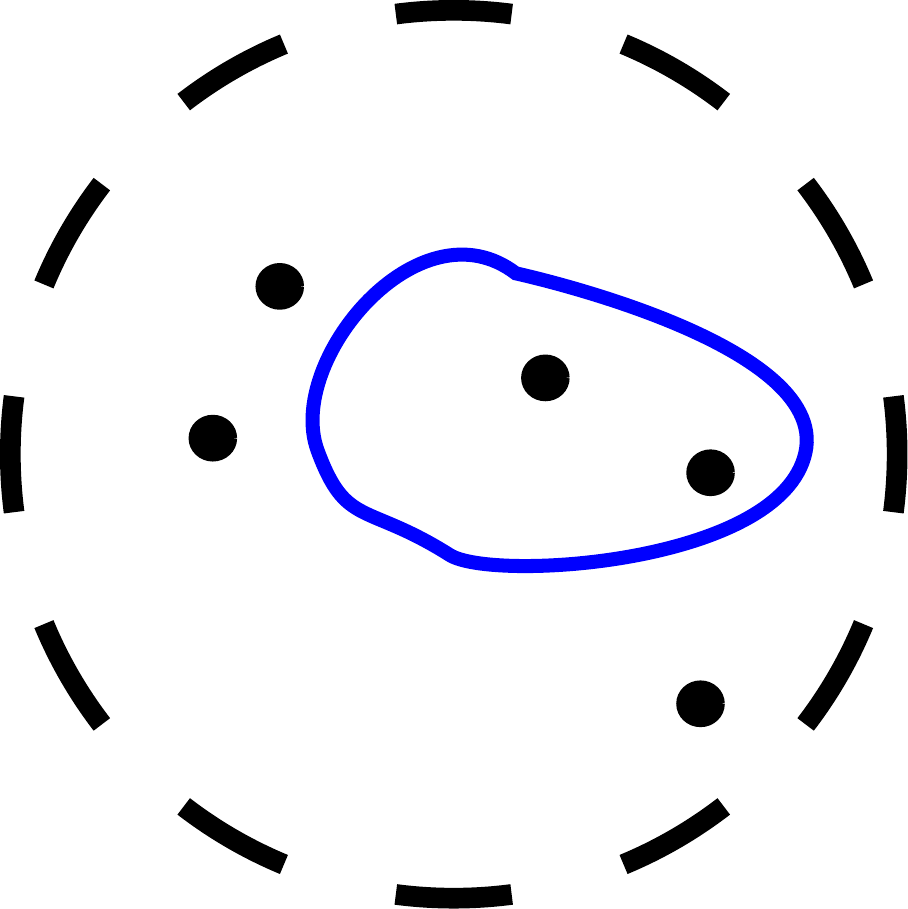}\end{minipage}$\;-\; q^{-1}$ \begin{minipage}{0.7in}\includegraphics[width=\textwidth]{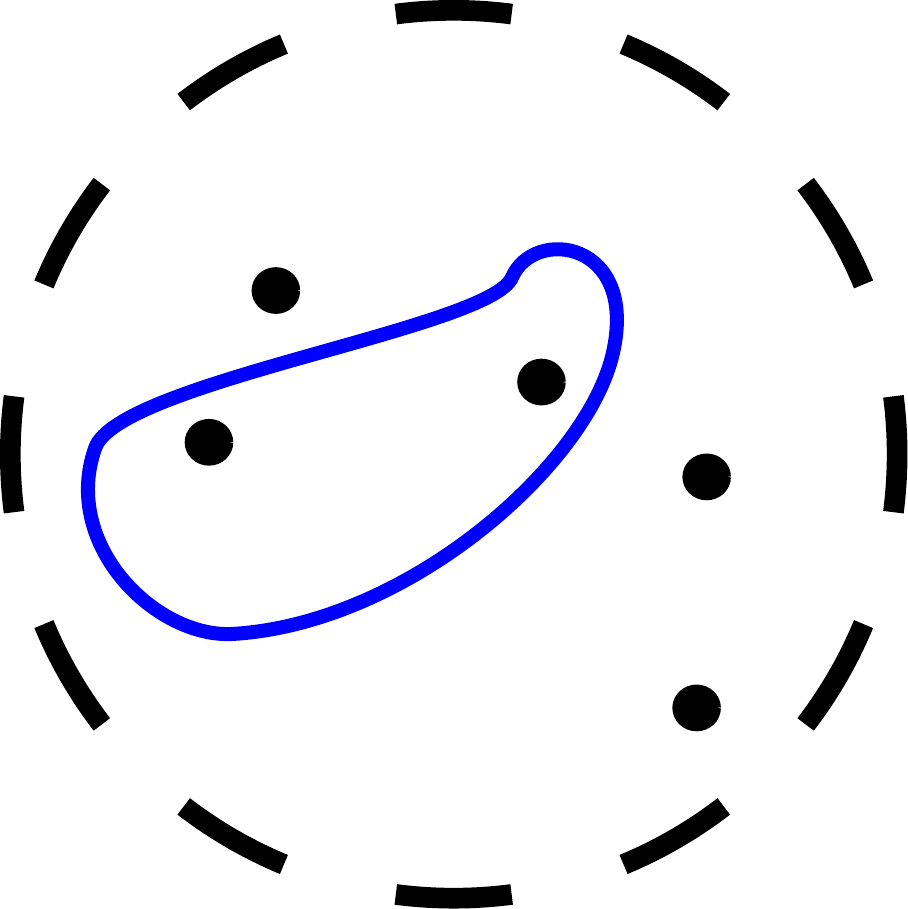}\end{minipage}$\;-\; $ \begin{minipage}{0.7in}\includegraphics[width=\textwidth]{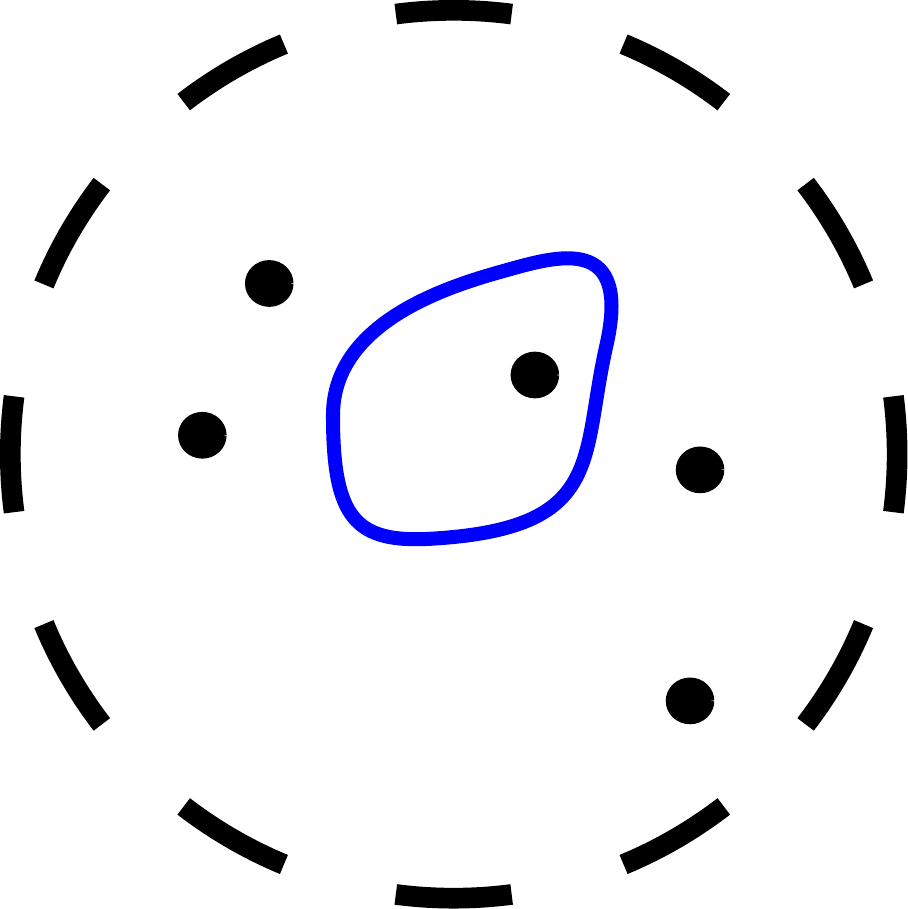}\end{minipage}
\caption{Generating a loop class by arc classes}
\label{fig:recursion}
\end{figure}

Now, we show that every reduced arc class is generated by $\{\beta_{ij}\}_{1 \le i < j \le n}$. Take a circle $D$, which properly contains $C$ (Figure \ref{fig:PDR}). Any multicurve in $\RR^{2}_{n}$ or $\Sigma_{0, n}$ is regular isotopic to a multicurve in $D$. To show this statement, we will apply skein relations in $D$. Thus, the proof will be identical in the case of $\RR^{2}_{n}$ and $\Sigma_{0, n}$. So from now on, we will focus on the $\RR^{2}_{n}$ case. 

Let $R := P \cup \bigcup_{i=1}^{n}R_{i}$, where $R_{i}$ is a ray emanating from each vertex $v_{i}$ toward the outer direction so that these rays are disjoint (Figure \ref{fig:PDR}). 

For each reduced arc $\alpha$, we define $i(\alpha, R)$ as the number of intersection points of the planar diagram of $\alpha$ and $R$, excluding intersections at vertices $v_{1}, v_{2}, \cdots, v_{n}$. Now, the \emph{intersection number} $\underline{i}(\alpha, R)$ is defined as 
\[
	\underline{i}(\alpha, R) := \mathrm{min}\;\{i(\alpha', R)\;|\; \alpha' \mbox{ is regular isotopic to } \alpha\}.
\]

We show the statement by induction on $\underline{i}(\alpha, R)$. If $\underline{i}(\alpha, R) = 0$, then $\alpha$ is regular isotopic to a reduced arc in $P$. Thus, $\alpha$ is regular isotopic to one of the $\beta_{ij}$'s. 

Suppose that $\underline{i}(\alpha, R) > 0$. In this case, $\alpha$ intersects one of $R_{i}$'s. Otherwise, $\alpha$ lies in a simply connected region $\RR^{2} \setminus \bigcup_{i=1}^{n}R_{i}$ and the entire curve $\alpha$ can be contracted to the interior of $P$. Thus, $\underline{i}(\alpha, R)$ must be zero. We may assume that $\alpha$ intersects $R_{i}$. Take the closest strand of $\alpha$ to $v_{i}$, and apply the Puncture-Skein relation at $v_{i}$. Then $\alpha = q^{\frac{1}{2}}v_{i}\gamma_{1}\gamma_{2} - q\delta$ as in Figure \ref{fig:intnumreduction}. Now $\underline{i}(\gamma_{1}, R), \underline{i}(\gamma_{2}, R), \underline{i}(\delta, R)$ are all strictly smaller than $\underline{i}(\alpha, R)$. By our induction hypothesis, we obtain the desired result.
\end{proof}

\begin{figure}[!ht]
$q^{-\frac{1}{2}}$\begin{minipage}{0.8in}\includegraphics[width=\textwidth]{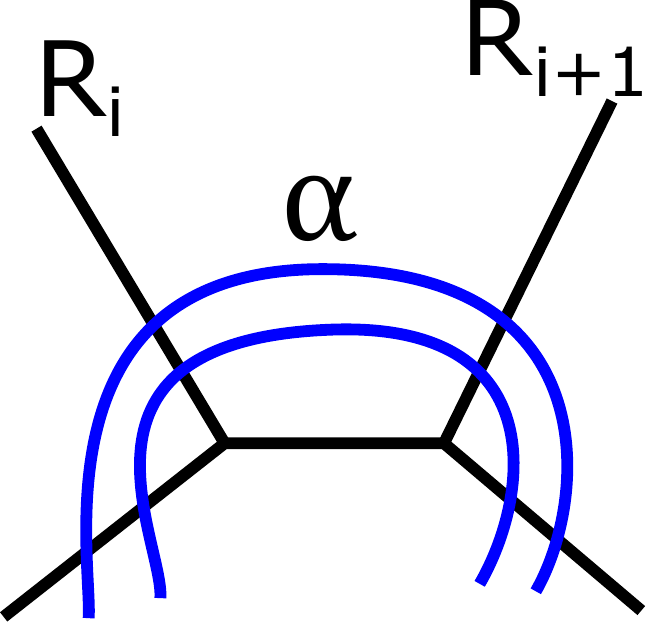}\end{minipage} $=$ \;$v_{i}$ \begin{minipage}{0.8in}\includegraphics[width=\textwidth]{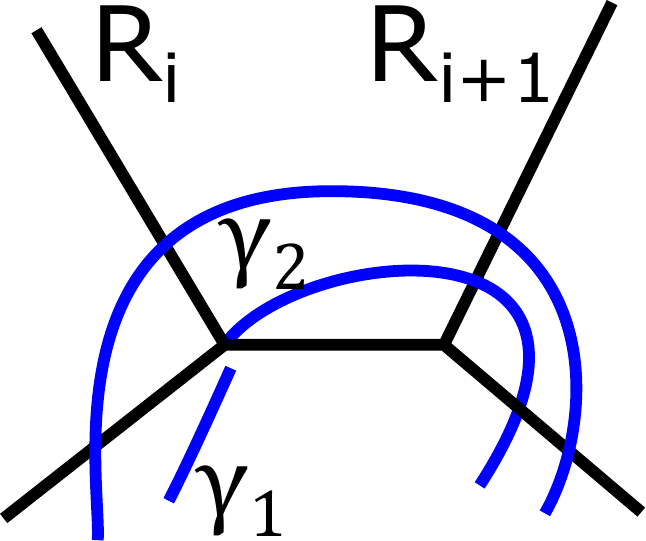}\end{minipage} $-\; q^{\frac{1}{2}}$ \begin{minipage}{0.8in}\includegraphics[width=\textwidth]{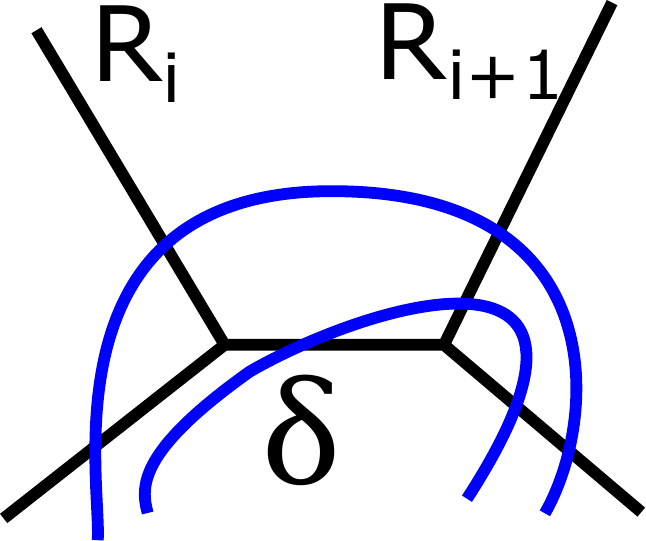}\end{minipage}
\caption{Intersection number reduction}
\label{fig:intnumreduction}
\end{figure}


\section{Relations}\label{sec:relations}

In this section, we present several geometric relations among curve classes on $\RR^{2}_{n}$ and $\Sigma_{0, n}$. The first two types of relations are valid for $\cA_{q}(\Sigma_{0, n})$ and $\cA_{q}(\RR^{2}_{n})$. The remaining relations are valid only for $\cA_{q}(\Sigma_{0, n})$.

\begin{definition}[Ptolemy relations]\label{def:Ptolemy}
Take a $4$-subset $I = \{i, j, k, \ell\} \subset [n]$ and assume that the four elements are listed in clockwise cyclic order. The \emph{Ptolemy relation} for $I$ is 
\begin{equation}\label{eqn:Ptolemy}
	\beta_{ik}\beta_{j\ell} = q\beta_{i\ell}\beta_{jk} + q^{-1}\beta_{ij}\beta_{k\ell}.
\end{equation}
\end{definition}

Note that this is a special case of the Skein relation in Definition \ref{def:skeinalgebra}.

\begin{definition}[Quantum commutation relations]\label{def:quantumcommutation}
For any $4$-subset $I = \{i, j, k, \ell\} \subset [n]$ that is listed in clockwise cyclic order, the \emph{first Quantum commutation relation} for $I$ is
\begin{equation}\label{eqn:quantcomm1}
	\beta_{ij}\beta_{k\ell} = \beta_{k\ell}\beta_{ij}.
\end{equation}
For any $3$-subset $I = \{i, j, k\} \subset [n]$ in clockwise cyclic order, the \emph{second Quantum commutation relation} for $I$ is 
\begin{equation}\label{eqn:quantcomm2}
	\beta_{jk}\beta_{ij} = q\beta_{ij}\beta_{jk} + (q^{-\frac{1}{2}}-q^{\frac{3}{2}})v_{j}^{-1}\beta_{ik}.
\end{equation}
\end{definition}

The second Quantum commutation relation follows from the comparison of the $\beta_{ij}\beta_{jk}$ and $\beta_{jk}\beta_{ij}$ after applying the Puncture-Skein relation at $v_{j}$. 

The below relations are valid only for $\cA_{q}(\Sigma_{0, n})$. 

\begin{definition}[$\gamma$-relations]\label{def:gammarelations}

Fix $i, j \in [n]$. Let $\gamma_{ij}^{+}$ (resp. $\gamma_{ij}^{-}$) be the reduced arc class outside $P$ in Figure \ref{fig:PDR}, starting from $v_{i}$, moving clockwise (resp. counterclockwise) and ending at $v_{j}$ (Figure \ref{fig:gammarelations}). It is clear that if $i \ne j$, $\gamma_{ij}^{+} = \gamma_{ji}^{-}$. When $i = j$, $\gamma_{ii}^{+}$ is the arc moving around $P$ while $\gamma_{ii}^{-} = \omega_{i}$ in Example \ref{ex:skeinexample}. The \emph{$\gamma$-relation} is  
\begin{equation}\label{eqn:gammaijgammaji}
	\gamma_{ij}^{+} = \gamma_{ij}^{-}. 
\end{equation}
\end{definition}

One can see that on $\Sigma_{0, n}$, $\gamma_{ij}^{+}$ and $\gamma_{ij}^{-}$ are regular isotopic. Note that $\gamma_{ii}^{-} = \omega_{i} = (q^{\frac{1}{2}}- q^{\frac{5}{2}})v_{i}^{-1} \in R_{q, n}$. If $j = i+1$, then $\gamma_{i,i+1}^{+} = \beta_{i,i+1}$.

\begin{figure}[!ht]
\begin{minipage}{0.8in}\includegraphics[width=\textwidth]{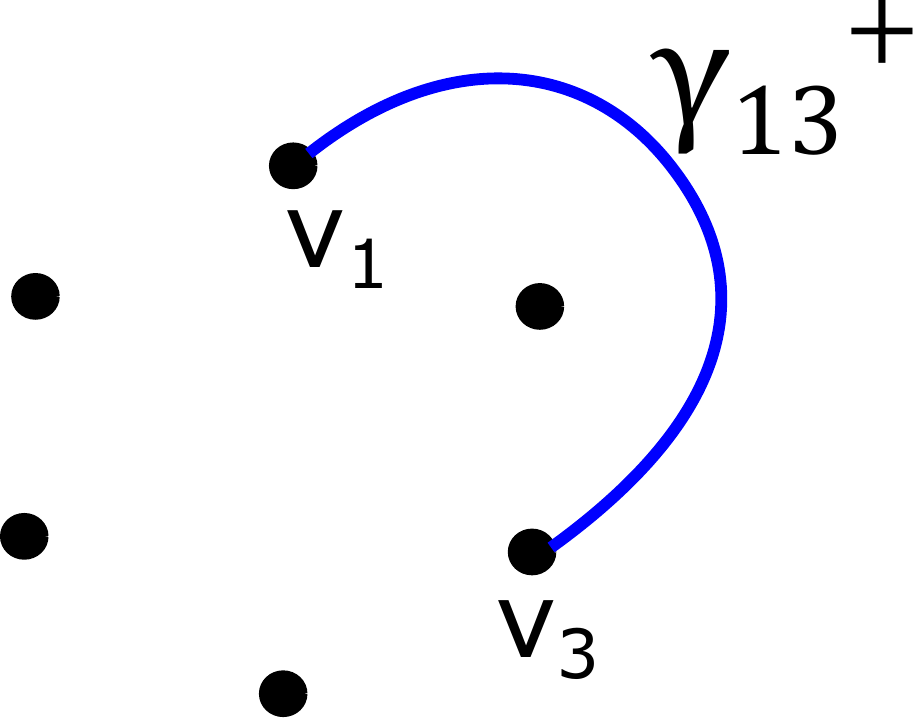}\end{minipage} $=$ \begin{minipage}{0.8in}\includegraphics[width=\textwidth]{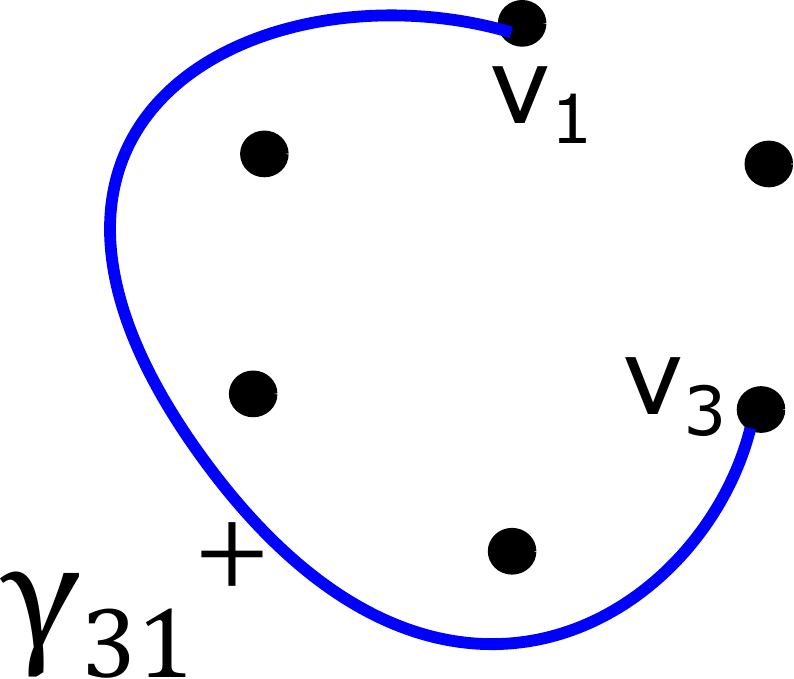}\end{minipage} \quad\quad \begin{minipage}{1in}\includegraphics[width=\textwidth]{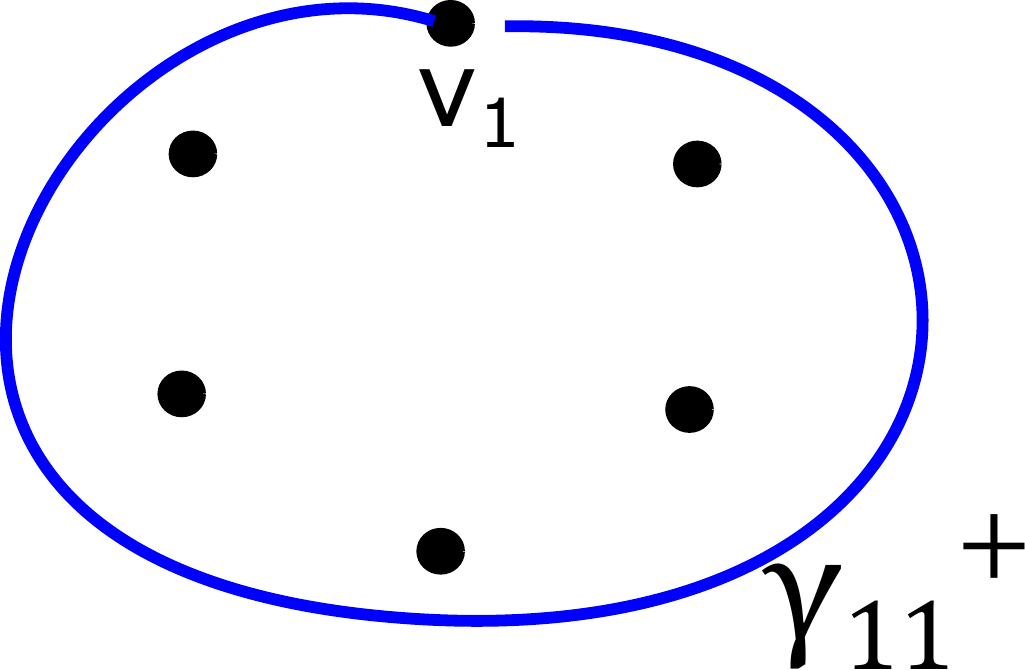}\end{minipage} $=$ \begin{minipage}{0.6in}\includegraphics[width=\textwidth]{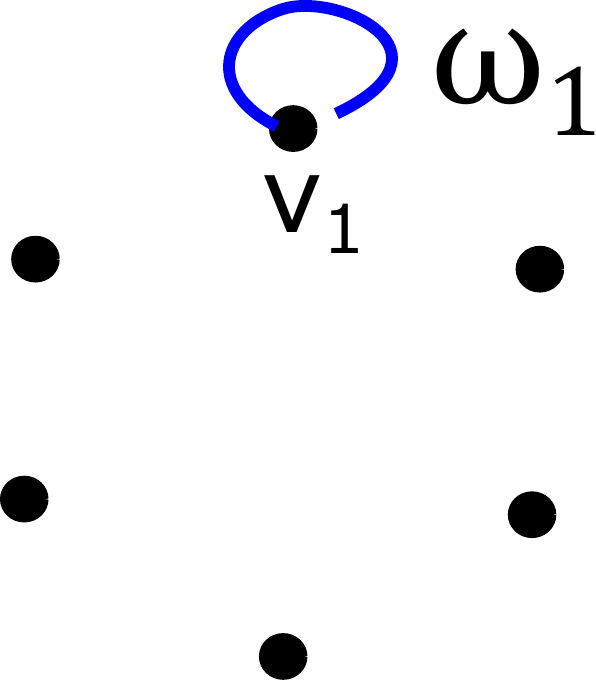}\end{minipage}
\caption{The $\gamma$-relations}
\label{fig:gammarelations}
\end{figure}

\begin{definition}[Big circle relation]\label{def:bigcircleistrivial}
Let $\delta$ be the reduced loop class of the circle $D$ in Figure \ref{fig:PDR}. The \emph{Big circle relation} is 
\begin{equation}\label{def:bigcircle}
	\delta = -q^{2} - q^{-2}.
\end{equation}
\end{definition}

The Big circle relation is a special case of the Framing relation in Definition \ref{def:skeinalgebra}.

By using skein relations, one may find explicit expressions for the relations above in terms of $\beta_{ij}$ classes. 

We introduce some new notations. Let $C_{n}$ be the cyclic graph with $n$ vertices $1, 2, \cdots, n$ arranged clockwise. For any two elements $i, j \in [n]$, let $(i,j)$ be the set of vertices in the path starting from $i+1$, moving clockwise, and ending at $j-1$. Note that $(i,i) = [n] \setminus \{i\} \ne \emptyset$. We may give a total order on $(i, j)$ as $i+1$ is the smallest and $j-1$ is the largest element. Then for any $I \subset (i, j)$, we have an induced order. 

\begin{lemma}\label{lem:gammarelationdescription}
For any $I \subset (i, j)$, let $i_{k}$ be the $k$-th element of $I$. Set $i_{0} = i$ and $i_{|I|+1} = j$. Then in $\cA_{q}(\Sigma_{0, n})$ or $\cA_{q}(\RR^{2}_{n})$, 
\begin{equation}\label{eqn:gammaclass}
	\gamma_{ij}^{+} = \gamma_{ji}^{-} = \sum_{I \subset (i, j)} (-1)^{|(i,j)\setminus I|}q^{|(i,j)| - \frac{|I|}{2}}\left(\prod_{k=1}^{|I|} v_{i_k}\right)\left(\prod_{k=0}^{|I|}\beta_{i_k i_{k+1}}\right).
\end{equation}
For $\gamma_{ii}^{+}$, we set $\beta_{ii} = \overline{\omega_{i}}$ (Example \ref{ex:conjugate}). 
\end{lemma}

The product $\mu_{I} := \prod_{k=0}^{|I|}\beta_{i_{k}i_{k+1}}$ is the ordered product of all edges on the clockwise path from $i$ to $j$ with intermediate vertex set $I$. The right hand side has a clear combinatorial meaning if we set $q = 1$ and ignore all vertex classes. It is an alternating sum of all clockwise paths from $v_{i}$ to $v_{j}$. 

\begin{proof}
For any $k \in (i, j)$, let $\eta_{ikj}$ be the reduced arc class starting at $v_{i}$, moving to the inside of $P$, going outside of $P$ between $v_{k-1}$ and $v_{k}$, moving clockwise, and arriving at $v_{j}$ (see Figure \ref{fig:etarelation}). Note that $\eta_{i,i+1,j} = \gamma_{ij}^{+}$. We set $\eta_{ijj} = \beta_{ij}$, and $\eta_{iii} = \beta_{ii} = \overline{\omega_{i}}$. 

\begin{figure}[!ht]
$v_{3}\;$\begin{minipage}{0.8in}\includegraphics[width=\textwidth]{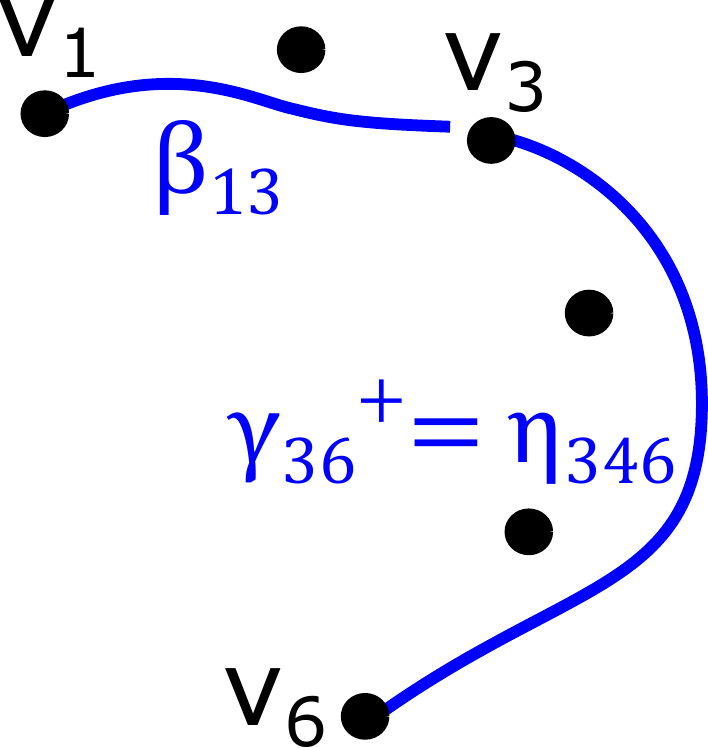}\end{minipage} $\;=\; q^{-\frac{1}{2}}\;$\begin{minipage}{0.8in}\includegraphics[width=\textwidth]{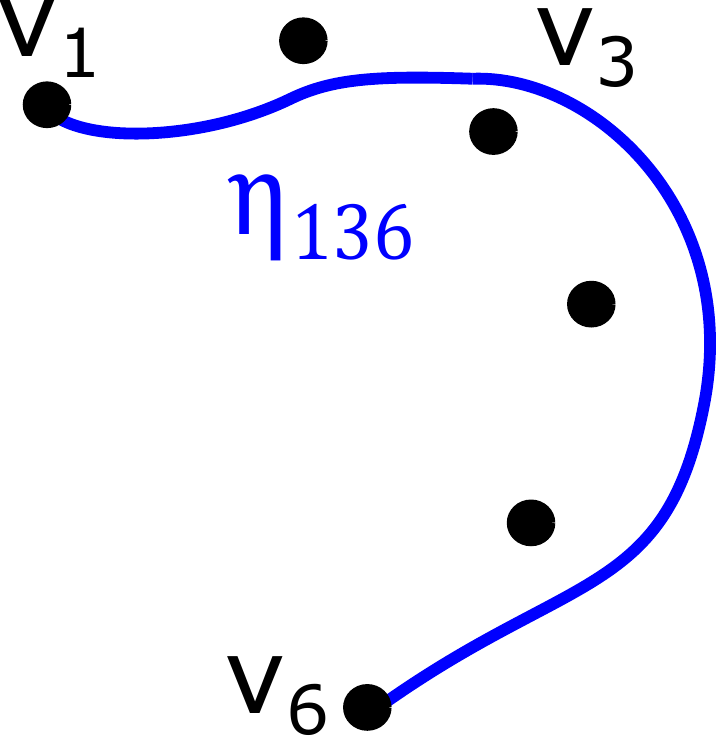}\end{minipage} $\;+ \; q^{\frac{1}{2}}\;$\begin{minipage}{0.8in}\includegraphics[width=\textwidth]{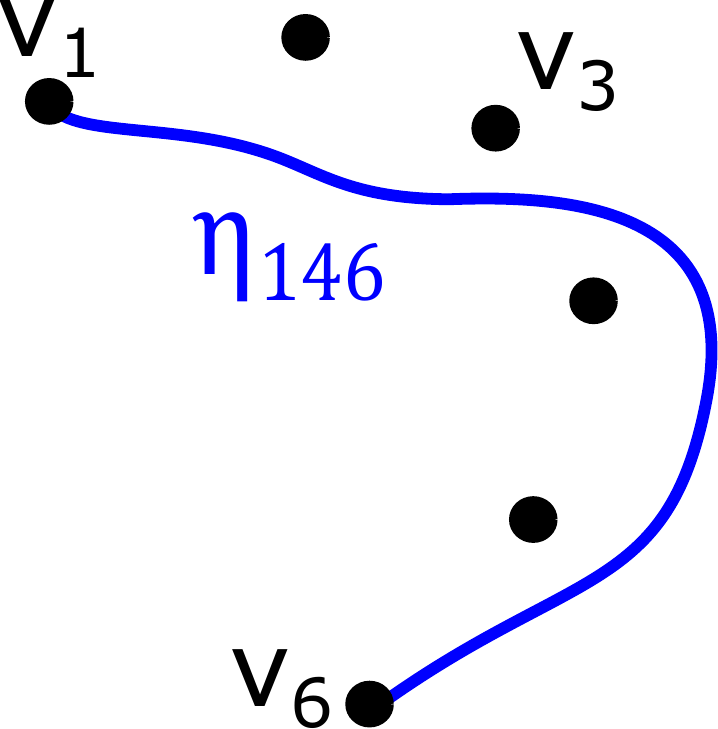}\end{minipage}
\caption{Recursive formula for $\gamma$-classes and $\eta$-classes}
\label{fig:etarelation}
\end{figure}

The Puncture-Skein relation at $v_{k}$ provides a recursive formula
\[
	\eta_{ikj} = q^{\frac{1}{2}}v_{k}\beta_{ik}\eta_{k,k+1,j} - q\eta_{i,k+1,j}
\]
(Figure \ref{fig:etarelation}). By applying the formula to $\gamma_{ij}^{+} = \eta_{i,i+1,j}$ and using $\eta_{ijj} = \beta_{ij}$, we obtain the desired result. 
\end{proof}

\begin{lemma}\label{lem:bigcirclerelation}
For any $I \subset [n]$, we denote the $k$-th element of $I$ by $i_{k}$ and set $i_{|I|+1} = i_{1}$. Then in $\cA_{q}(\Sigma_{0, n})$ or $\cA_{q}(\RR^{2}_{n})$, 
\begin{equation}\label{eqn:deltaformula}
	\delta = (-1)^{n-1}(q^{n-2}+(q^{-1})^{n-2}) + \sum_{\stackrel{I \subset [n]}{|I| \ge 2}} (-1)^{n - |I|}q^{n-2i_{1}+1 - \frac{|I|}{2}}\left(\prod_{i \in I}v_{i}\right)\left(\prod_{k=1}^{|I|}\beta_{i_{k}i_{k+1}}\right).
\end{equation}
\end{lemma}

Note that the product $\mu_{I} := \prod_{k=1}^{|I|}\beta_{i_{k}i_{k+1}}$ is the ordered product of all edges of the convex polygon with the vertex set $I$. The non-constant part of the right hand side is, after setting $q = 1$ and ignoring vertices, the alternating sum of all convex polygons in $P$, including bigons. 

\begin{proof}
Here we leave an outline of the proof. For any $I \subset [n]$, let $\nu_{I}$ be the product of $\prod_{i\in I}v_{i}$ and the configuration of curves $\{\beta_{i_{k}i_{k+1}}\}_{1 \le k \le |I|}$ that is like `infinite stairs' in Figure \ref{fig:muandnu}. Note that $\nu_{I} \ne \mu_{I}$ because in $\nu_{I}$, there is no lowest component. We set $\mu_{\{i\}} = v_{i}\overline{\omega_{i}}$, $\nu_{\{i\}} = v_{i}\omega_{i}$, and $\mu_{\emptyset} = \nu_{\emptyset} = -q^{2} - q^{-2}$.

\begin{figure}[!ht]
\begin{minipage}{0.8in}\includegraphics[width=\textwidth]{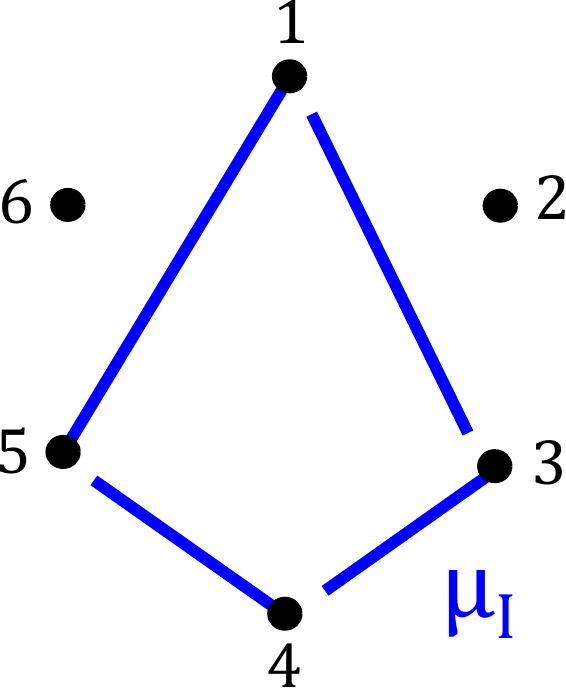}\end{minipage} \quad \quad \begin{minipage}{0.8in}\includegraphics[width=\textwidth]{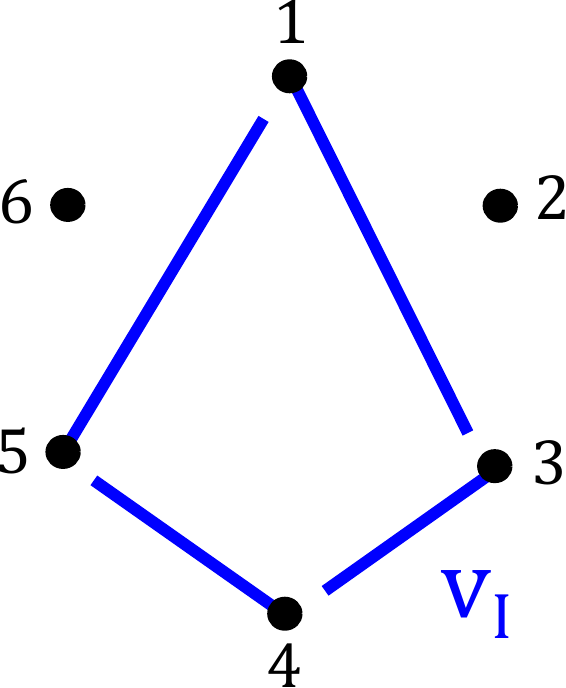}\end{minipage}
\caption{Example of $\mu_{I}$ and $\nu_{I}$ for $I = \{1, 3, 4, 5\}$.}
\label{fig:muandnu}
\end{figure}

For any $k \ge 1$, let $I_{\ge k} := \{i_{k}, i_{k+1}, \cdots, i_{|I|}\} \subset I$. Applying the Puncture-Skein relation, we get a recursive formula $\nu_{I} = q^{-1}\mu_{I} + (q^{\frac{1}{2}} - q^{-\frac{3}{2}})\mu_{I_{\ge 2}}$. Solving the recursive equation, we have 
\begin{equation}\label{eqn:nubymu}
	\nu_{I} = (q^{\frac{1}{2}} - q^{-\frac{3}{2}})^{|I|-1}\omega_{i_{|I|}}+\sum_{j=1}^{|I|-1}q^{-1}(q^{\frac{1}{2}}-q^{-\frac{3}{2}})^{j-1}\mu_{I_{\ge j}}.
\end{equation}

On the other hand, let $\delta_{I}$ be the loop class properly containing the convex hull generated by $\{v_{i}\}_{i \in I}$. So, $\delta = \delta_{[n]}$. Applying the Puncture-Skein relation at all vertices, we obtain $\nu_{I} = \sum_{J \subset I}q^{\frac{|I|}{2}-|J|}\delta_{J}$. Using the M\"obius inversion formula (\cite[Section 3.7]{Stanley12}), we get
\begin{equation}\label{eqn:inversion}
	\delta_{I} = \sum_{J \subset I}(-1)^{|I| - |J|}q^{\frac{|J|}{2} - |I|}\nu_{J}.
\end{equation}
Combining \eqref{eqn:nubymu} and \eqref{eqn:inversion}, we may describe $\delta = \delta_{[n]}$ as a linear combination of $\mu_{J}$'s and $\omega_{i}$'s. By calculating each coefficient, we obtain the formula. 
\end{proof}

\begin{remark}\label{rem:oppositewaterdrop}
A careful reader may wonder if we require an extra relation $\overline{\gamma_{ii}^{+}} = \overline{\gamma_{ii}^{-}}$. This extra relation follows from $\gamma_{ii}^{+} = \gamma_{ii}^{-}$ by applying the conjugation map, which is an anti-involution. It can be also obtained by combining the $\gamma$-relations, the Big circle relation, the Puncture-Skein relation, and the computation in Example \ref{ex:conjugate}, as follows:
\[
	\overline{\gamma_{ii}^{+}} = q\gamma_{ii}^{+} + (q^{-\frac{1}{2}} - q^{\frac{3}{2}})v_{i}^{-1}\delta = q\omega_{i} + (q^{-\frac{1}{2}} - q^{\frac{3}{2}})v_{i}^{-1}(-q^{2}-q^{-2}) = \overline{\omega_{i}} = \overline{\gamma_{ii}^{-}}.
\]
\end{remark}


\section{Presentation of $\cA_{q}(\RR^{2}_{n})$}\label{sec:presentationR2}

From now on, we assume that $n \ge 3$. In this section, we find a presentation of $\cA_{q}(\RR^{2}_{n})$. This computation is not only interesting but is also a crucial step for the calculation of $\cA_{q}(\Sigma_{0, n})$ because of the existence of functorial morphisms. 

Since $\RR^{2} \cong S^{2} \setminus \{p\}$, there is a natural inclusion map $\iota : \RR^{2}_{n} \to \Sigma_{0, n}$ which maps the $i$-th vertex to the $i$-th vertex. It induces an $R_{q, n}$-algebra homomorphism
\begin{equation}\label{eqn:imap}
	\iota_{\#} : \cA_{q}(\RR^{2}_{n}) \to \cA_{q}(\Sigma_{0, n}), 
\end{equation}
which is specialized to $\iota_{\#}: \cC(\RR^{2}_{n}) \to \cC(\Sigma_{0, n})$. To avoid any unnecessary complication of notations, we will use the same symbol $\iota_{\#}$ for the classical case and for the map after the coefficient extension. 

We may also regard $p \in S^{2}$ as the $(n+1)$-st puncture on  $S^{2}$. There is another morphism $j : \RR^{2}_{n} \to \Sigma_{0, n+1}$ that maps the $i$-th vertex to the $i$-th vertex for $1 \le i \le n$. Then we obtain an $R_{q, n}$-algebra homomorphism 
\begin{equation}\label{eqn:jmap}
	j_{\#} : \cA_{q}(\RR^{2}_{n}) \to \cA_{q}(\Sigma_{0, n+1}).
\end{equation}

\begin{definition}
Let $S_{q, n}$ be the non-commutative $R_{q, n}$-algebra generated by $\beta_{ij}$ classes modulo the ideal generated by relations in Definitions \ref{def:Ptolemy} and \ref{def:quantumcommutation}. 
\end{definition}

Consider the `classical limit' $S_{q, n}/(q^{\frac{1}{2}}-1)$ of $S_{q, n}$. Then the quantum commutation relations specialize to ordinary commutation relations, and one can check that 
\begin{equation}\label{eqn:classicallimit}
	S_{q, n}/(q^{\frac{1}{2}}-1) \cong R_{n} \otimes_{\ZZ}S, 
\end{equation}
where $R_{n} = \ZZ[v_{1}^{\pm}, v_{2}^{\pm}, \cdots, v_{n}^{\pm}]$ and $S$ is a commutative algebra with the presentation 
\begin{equation}\label{eqn:defS}
	S := \ZZ[\beta_{ij}]/(\beta_{ik}\beta_{j\ell} - \beta_{ij}\beta_{k\ell} - \beta_{i\ell}\beta_{jk})
\end{equation}
where $\{i, j, k, \ell\} \subset [n]$ is cyclic. 

\begin{remark}\label{rem:ringS}
The commutative ring $S$ in \eqref{eqn:defS} has appeared in many different contexts.
\begin{enumerate}
\item The algebra $\CC \otimes_{\ZZ} S$ is the homogeneous coordinate ring of the Grassmannian $\mathrm{Gr}(2, n)$ (\cite[Chapter I.5]{GriffithsHarris94}).
\item In classical invariant theory, $S$ is the $\mathrm{SL}_{2}(\ZZ)$-invariant subring of the algebra of the polynomial ring with $2n$ variables. It is also called the \emph{graphical algebra} \cite[Section 2]{MoonSwinarski19}).
\item The algebra $S$ is also a cluster algebra of type $A_{n-3}$ (\cite[Section 12]{FominZelevinsky03}). 
\item The tropicalization of $\CC \otimes_{\ZZ}S$ is the space of phylogenetic trees in computational biology (\cite[Sec 4.3]{MaclaganSturmfels15}). 
\end{enumerate}
In particular, it is well-known that $\CC \otimes_{\ZZ}S$ is an integral domain of dimension $\dim \mathrm{Gr(2,n)} + 1 = 2n - 3$. Thus, $\dim \CC \otimes_{\ZZ} R_{n} \otimes_{\ZZ} S = 3n - 3$. 
\end{remark}

\begin{remark}\label{rem:basis}
We say that a monomial $\prod \beta_{ij}^{m_{ij}}$ is \emph{non-crossing} if no two $\beta_{ij}$ and $\beta_{k\ell}$ with $m_{ij}, m_{k\ell} > 0$ intersect except at one of the endpoints. The ring $S$ is a free $\ZZ$-module with a basis consisting of non-crossing monomials with respect to the $\beta_{ij}$'s (Straightening law, \cite[Corollary 3.1.9]{Sturmfels08}). Since the freeness is preserved by the base ring extension, $R_{n} \otimes_{\ZZ}S$ is a free $R_{n}$-module with the same basis. 

We will extend this freeness result to the quantum setup. Fix any total order on the set $\{\beta_{ij}\}$. Let $B$ be the set of non-crossing monomials with respect to $\beta_{ij}$ in $S_{q, n}$, such that the product is taken in non-decreasing order. By the non-decreasing property, there is no duplication of the commutative version of monomials in $B$. Thus, $B$ can be also understood as a basis of $S$ and $R_{n} \otimes_{\ZZ}S$. In Proposition \ref{prop:quantumbasis}, we will show that $S_{q, n}$ is a free $R_{q, n}$ module with a basis $B$. 
\end{remark}

Since the non-commutative polynomial algebra is a free object in the category of algebras, there is a unique homomorphism 
\[
	f : R_{q, n}\langle \beta_{ij}\rangle  \to \cA_{q}(\RR^{2}_{n})
\]
which maps each $\beta_{ij}$ to $\beta_{ij} \in \cA_{q}(\RR^{2}_{n})$. Note that Ptolemy relations and Quantum commutation relations are special cases of Skein relations and Puncture-Skein relations, respectively. Thus, there is a well-defined quotient homomorphism 
\begin{equation}\label{eqn:barf}
	\bar{f} : S_{q, n} \to \cA_{q}(\RR^{2}_{n}).
\end{equation}
By Proposition \ref{prop:generators}, $\bar{f}$ is surjective. When $q^{\frac{1}{2}} = 1$, this map is specialized to (by abuse of notation, we use the same letter) a surjective homomorphism 
\begin{equation}\label{eqn:barfclassical}
	\bar{f} : R_{n} \otimes_{\ZZ}S \to \cC(\RR^{2}_{n}).
\end{equation}

The main result of this section is the following. 

\begin{theorem}\label{thm:presentationAR2n}
The $R_{q, n}$-algebra homomorphism $\bar{f} : S_{q, n} \to \cA_{q}(\RR^{2}_{n})$ in \eqref{eqn:barf} is an isomorphism.
\end{theorem}

We prove the classical case first. 

\begin{theorem}\label{thm:presentationCR2n}
The $R_{n}$-algebra homomorphism $\bar{f} : R_{n} \otimes_{\ZZ} S \to \cC(\RR^{2}_{n})$ in \eqref{eqn:barfclassical} is an isomorphism.
\end{theorem}

\begin{remark}\label{rem:interpretationofS}
Theorem \ref{thm:presentationCR2n} provides a skein theoretic interpretation of $S$. 
\end{remark}

\begin{proof}[Proof of Theorem \ref{thm:presentationCR2n}]

It is sufficient to show the injectivity of $\bar{f}$. We need the fact that $\dim \CC \otimes_{\ZZ}\cC(\RR^{2}_{n}) = 3n - 3$. This will be shown in Proposition \ref{prop:dimCR2n} below. By assuming it here, we will prove the injectivity. Suppose that $\bar{f}$ is not injective. Then, there is an isomorphism $R_{n}\otimes_{\ZZ}S/\ker \bar{f} \cong \cC(\RR^{2}_{n})$. 

Note that for any integral domain $A$ and an indeterminate $v$, $A[v^{\pm}]$ is also an integral domain. Since $S$ is an integral domain, we can conclude that $R_{n} \otimes_{\ZZ}S \cong S[v_{1}^{\pm}, v_{2}^{\pm}, \cdots, v_{n}^{\pm}]$ is an integral domain too. Now every nonzero element in $\ker \bar{f}$ is not a zero divisor. Thus 
\[
	3n - 3 = \dim \CC \otimes_{\ZZ}R_{n} \otimes_{\ZZ}S > \dim \CC \otimes_{\ZZ}R_{n} \otimes_{\ZZ}S /\ker \bar{f} = \dim \CC \otimes_{\ZZ}\cC(\RR^{2}_{n}) = 3n - 3,
\]
which is a contradiction.
\end{proof}

The next proposition fills the missing step of dimension computation.

\begin{proposition}\label{prop:dimCR2n}
The dimension of $\CC \otimes_{\ZZ}\cC(\RR^{2}_{n})$ is $3n - 3$. 
\end{proposition}

\begin{proof}
Since there is a surjective homomorphism $f : \CC \otimes_{\ZZ}R_{n} \otimes_{\ZZ} S  \to \CC \otimes_{\ZZ}\cC(\RR^{2}_{n})$, we know that $\dim \CC \otimes_{\ZZ}\cC(\RR^{2}_{n}) \le \dim \CC \otimes_{\ZZ} R_{n} \otimes_{\ZZ} S = 3n - 3$ by Remark \ref{rem:ringS}. 

Consider the map $j_{\#} : \CC \otimes_{\ZZ}\cC(\RR^{2}_{n}) \to \CC \otimes_{\ZZ}\cC(\Sigma_{0, n+1})$ which comes from $j_{\#}$ in \eqref{eqn:jmap} by the coefficient extension. Let $M$ be the image of $j_{\#}$. The homomorphism $j_{\#}$ is not injective. Indeed, the loop class $\delta$ enclosing the polygon $P$ (the isotopy class of the circle $D$ in Figure \ref{fig:PDR}) is also isotopic to the punctured circle at $v_{n+1}$ on $\Sigma_{0, n+1}$. In other words, $\delta - 2 \in \ker j_{\#}$. Thus, there are well-defined surjective homomorphisms 
\[
	\CC \otimes_{\ZZ}\cC(\RR^{2}_{n}) \to \CC \otimes_{\ZZ}\cC(\RR^{2}_{n})/(\delta - 2) \to M.
\]
Therefore, $\dim \CC \otimes_{\ZZ}\cC(\RR^{2}_{n}) \ge \dim \CC \otimes_{\ZZ}\cC(\RR^{2}_{n})/(\delta - 2) \ge \dim M$. The first inequality is an equality only if $\delta - 2$ is a zero divisor in $\CC \otimes_{\ZZ}\cC(\RR^{2}_{n})$. 

Suppose that $(\delta - 2)h = 0$ for some nonzero $h \in \CC \otimes_{\ZZ}\cC(\RR^{2}_{n})$. By Proposition \ref{prop:freeness}, we may represent $h$ uniquely as a linear combination of reduced multicurves. Furthermore, we may find representatives of those reduced multicurves that are contained in $D$. So all of them are disjoint from $\delta$. Thus, $0 = (\delta - 2)h$ is a nontrivial linear combination of reduced multicurves, which violates the freeness of $\cC(\RR^{2}_{n})$ (Proposition \ref{prop:freeness}). Therefore, $\delta - 2$ is not a zero divisor in $\CC \otimes_{\ZZ}\cC(\RR^{2}_{n})$. Thus, $\dim \CC \otimes_{\ZZ}\cC(\RR^{2}_{n}) > \dim M$. 

Since $\CC \otimes_{\ZZ}\cC(\Sigma_{0, n+1})$ is an integral domain by Theorem \ref{thm:domain}, $M$ is also an integral domain. So $\dim M$ is equal to the transcendental degree of the field of fractions $Q(M)$. Let $Q(M)(v_{n+1})$ be an extension field of $Q(M)$ by the vertex class $v_{n+1}$. We claim that $Q(\CC \otimes_{\ZZ}\cC(\Sigma_{0, n+1}))$ is a finite extension of $Q(M)(v_{n+1})$. By Proposition \ref{prop:dimension}, $Q(\CC \otimes_{\ZZ}\cC(\Sigma_{0, n+1}))$ is generated by the curve classes for edges in a triangulation of $\Sigma_{0, n+1}$. 

Fix a triangulation $T'$ of $P$. By adding $n$ rays $e_{i}$ that each connect $v_{i}$ and $v_{n+1}$, we can make a triangulation $T$ of $S^{2}$ with $n+1$ vertices. We use it to construct a transcendental basis of $Q(\CC \otimes_{\ZZ}\cC(\Sigma_{0, n+1}))$. Note that all edges from $T'$ are already in $Q(M)$. 

For each $v_{i}$, consider an arc class $\alpha_{i}$ that starts from $v_{i}$, moves around $v_{n+1}$ clockwise, and comes back to $v_{i}$ as in Figure \ref{fig:alpha}. Then, $\alpha_{i} \in \mathrm{im}\; j_{\#} = M$. By applying the Puncture-Skein relation at $v_{n+1}$ and utilizing the fact that $\omega_{i} = 0$ when $q^{\frac{1}{2}} = 1$ (Example \ref{ex:skeinexample}), we obtain $\alpha_{i} = v_{n+1}e_{i}^{2}$. Thus, $e_{i}^{2} \in Q(M)(v_{n+1})$ and $e_{i}$ is in the finite extension of $Q(M)(v_{n+1})$. Therefore, $Q(\CC \otimes_{\ZZ}\cC(\Sigma_{0, n+1}))$ is a finite extension of $Q(M)(v_{n+1})$. 

\begin{figure}[!ht]
\begin{minipage}{0.8in}\includegraphics[width=\textwidth]{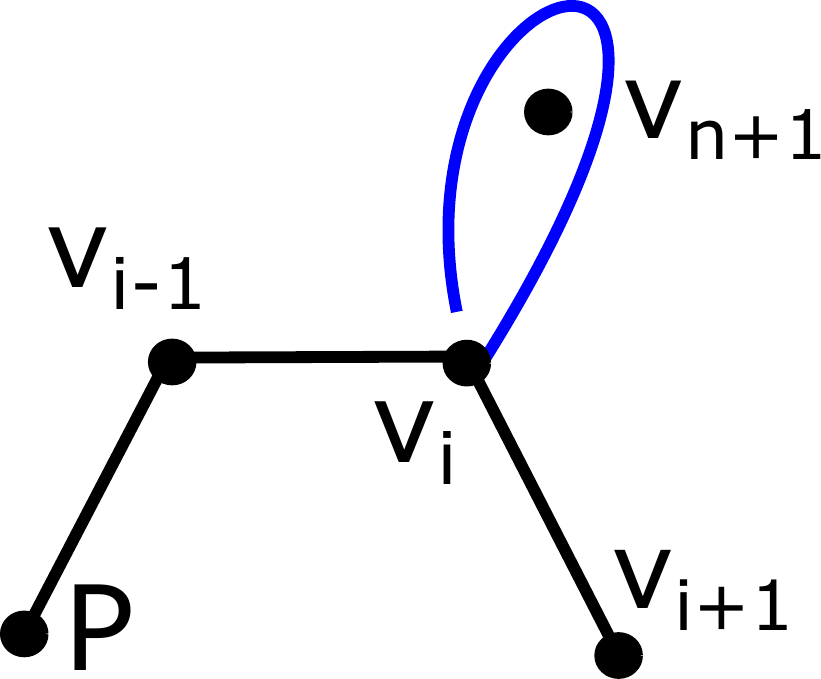}\end{minipage}
\caption{The curve class $\alpha_{i}$}
\label{fig:alpha}
\end{figure}

Since a finite extension does not change the transcendental degree, 
\[
\begin{split}
	\dim \CC\otimes_{\ZZ}\cC(\Sigma_{0,n+1}) &= \mathrm{tr.deg} \;Q(\CC \otimes_{\ZZ}\cC(\Sigma_{0, n+1})) = \mathrm{tr.deg} \;Q(M)(v_{n+1})\\
	&\le \mathrm{tr.deg}\;Q(M) + 1 = \dim M + 1.
\end{split}
\]
Thus, $3n - 4 \le \dim M < \dim \CC \otimes_{\ZZ}\cC(\RR^{2}_{n}) \le 3n - 3$. Therefore, $\dim \CC \otimes_{\ZZ}\cC(\RR^{2}_{n}) = 3n - 3$. 
\end{proof}

Now consider the quantum case. Note that there is a commutative diagram 
\begin{equation}\label{eqn:commdiagram}
	\xymatrix{
	S_{q, n} \ar[r]^{\bar{f}} \ar[d]_{\pi = /(q^{\frac{1}{2}}-1)} & \cA_{q}(\RR^{2}_{n}) \ar[d]^{\rho = /(q^{\frac{1}{2}}-1)}\\
	R_{n} \otimes_{\ZZ}S \ar[r]^{\bar{f}} & \cC(\RR^{2}_{n}).}
\end{equation}

\begin{lemma}\label{lem:notorsion}
If $g \in S_{q, n}$ is nonzero, then $(q^{\frac{1}{2}}-1)^{k}g \ne 0$ for all $k > 0$.
\end{lemma}

\begin{proof}
Suppose that $(q^{\frac{1}{2}}-1)^{k}g = 0$. Since $\bar{f} : S_{q, n} \to \cA_{q}(\RR^{2}_{n})$ is an $R_{q, n}$ module homomorphism, $\bar{f}((q^{\frac{1}{2}}-1)^{k}g) = (q^{\frac{1}{2}}-1)^{k}\bar{f}(g) = 0$. Since $\cC(\RR^{2}_{n}) \cong R_{n} \otimes_{\ZZ}S$ is an integral domain (Theorem \ref{thm:presentationCR2n}), by Theorem \ref{thm:classicalimpliesquantum}, $\cA_{q}(\RR^{2}_{n})$ is also a domain. Thus $\bar{f}(g) = 0$ and $g \in \ker \bar{f}$ (It also follows from the freeness in Proposition \ref{prop:freeness}.). Then $g\in \ker \pi$, as $(\bar{f} \circ \pi)(g) = (\rho \circ \bar{f})(g) = 0$ and $\bar{f} : R_{n} \otimes_{\ZZ}S \to \cC(\RR^{2}_{n})$ is an isomorphism. 

By applying Ptolemy relations, we know that every element in $S_{q, n}$ can be written as a $R_{q, n}$-linear combination of non-crossing monomials. By the Quantum commutation relations and induction on the total degree, every element can be written as an $R_{q, n}$-linear combination of monomials in $B$. In particular, we may write $g = \sum c_{I}\beta_{I}$ as a linear combination of monomials in $B$. Then $0 = \pi(g) = \sum \bar{c}_{I}\beta_{I}$, where $\bar{c}_{I}$ is the image of $c_{I}$ by the map $R_{q, n} \to R_{n}$ sending $q^{\frac{1}{2}}$ to one. Since $R_{n} \otimes_{\ZZ}S$ is a free $R_{n}$-module with a basis $B$, $\bar{c}_{I} = 0$ for all $I$. In other words, $c_{I}$ is a multiple of $q^{\frac{1}{2}}-1$. Therefore, $g = \sum c_{I}\beta_{I} = \sum (q^{\frac{1}{2}}-1)c_{I}'\beta_{I} = (q^{\frac{1}{2}}-1)g'$ for some $g' \in S_{q, n}$. But now $(q^{\frac{1}{2}}-1)^{k+1}g' = (q^{\frac{1}{2}}-1)^{k}g = 0$, so by the same argument we can continue to factor $g'$. However, this procedure must be terminated as $\cA_{q}(\RR^{2}_{n})$ is a free $R_{q, n}$-module, so it is not divisible. Thus, we have a contradiction. 
\end{proof}

\begin{proposition}\label{prop:quantumbasis}
The algebra $S_{q, n}$ is a free $R_{q, n}$-module with a basis $B$.
\end{proposition}

\begin{proof}
As before, every element can be written as an $R_{q, n}$-linear combination of monomials in $B$. Suppose there is a nontrivial relation $\sum c_{I}\beta_{I} = 0$ where $\beta_{I} \in B$. If all of $c_{I}$ has $q^{\frac{1}{2}}-1$ as a common factor, then we may write $\sum c_{I}\beta_{I} = (q^{\frac{1}{2}}-1)\sum c_{I}'\beta_{I}$. Then $\sum c_{I}'\beta_{I} = 0$ due to Lemma \ref{lem:notorsion}. By dividing the relation by an appropriate power of $q^{\frac{1}{2}}-1$, we may assume that the coefficients of $\sum c_{i}\beta_{I}$ do not have $q^{\frac{1}{2}}-1$ as a common factor. Then, by setting $q^{\frac{1}{2}} = 1$, we obtain a relation $\sum \bar{c}_{I}\beta_{I} = 0$ in $R_{n} \otimes_{\ZZ} S$. Because one of the $c_{I}$'s does not have $q^{\frac{1}{2}}-1$ as a factor, $\bar{c}_{I}$ is nonzero,  violating the freeness of $R_{n} \otimes_{\ZZ}S$. 
\end{proof}

\begin{proof}[Proof of Theorem \ref{thm:presentationAR2n}]
Pick a nonzero $g \in \ker \bar{f}$ for $\bar{f} : S_{q, n} \to \cA_{q}(\RR^{2}_{n})$. Since $S_{q, n}$ is a free $R_{q, n}$-module (Proposition \ref{prop:quantumbasis}), we may express $g = \sum c_{I}\beta_{I}$ as an $R_{q, n}$-linear combination of monomials in $B$ uniquely. As in the proof of Proposition \ref{prop:quantumbasis}, after dividing $g$ by an appropriate power of $q^{\frac{1}{2}} - 1$, we may assume that $\bar{g}$ is nonzero in $R_{n} \otimes_{\ZZ} S$. But from $\bar{f}(\bar{g}) = 0$ and the injectivity in the classical case, we obtain a contradiction. 
\end{proof}


\section{Presentation of $\cA_{q}(\Sigma_{0,n})$}\label{sec:presentationS0n}

Now we restate our main theorem and give the proof.

\begin{theorem}\label{thm:matinthm}
Let $J$ be the ideal generated by the Ptolemy relations, the Quantum commutation relations, the $\gamma$-relations, and the Big circle relation in Section \ref{sec:relations}. Then 
\[
	\cA_{q}(\Sigma_{0, n}) \cong R_{q, n} \langle \beta_{ij}\rangle/J.
\]
\end{theorem}

\begin{proof}
By Theorem \ref{thm:presentationAR2n}, it is sufficient to show that $\cA_{q}(\RR^{2}_{n})/K \cong \cA_{q}(\Sigma_{0, n})$, where $K$ is the ideal generated by the $\gamma$-relations and the Big circle relation in Definitions \ref{def:gammarelations} and \ref{def:bigcircleistrivial}. 

Recall that there is a functorial morphism $\iota_{\#} : \cA_{q}(\RR^{2}_{n}) \to \cA_{q}(\Sigma_{0, n})$. This map is surjective because any regular isotopy class of a multicurve in $\Sigma_{0, n}$ can be represented by a multicurve in $\RR^{2}_{n}$ by avoiding the point $p \in S^{2} \setminus \RR^{2}$. Thus, it is sufficient to show that $\ker \iota_{\#} = K$. It is clear that $K \subset \ker \iota_{\#}$. 

To find the extra relations, we use an argument imitating the `handle slide lemma' in \cite[Proposition 2.2.(5)]{Przytycki99} (see also \cite[Section 3]{BullockLoFaro05}). Observe that two non-isotopic curves $\alpha_{1}$ and $\alpha_{2}$ in $\RR^{2}_{n}$ can be isotopic in $\Sigma_{0, n}$ (i.e., $\iota_{\#}(\alpha_{1}) = \iota_{\#}(\alpha_{2})$) because in $\Sigma_{0, n}$, some strands of $\alpha_{1}$ can freely cross $p \in S^{2} \setminus \RR^{2}$. Indeed, this is the only reason for the difference in the regular isotopy classes in these two surfaces. This is because if we fix two points on the dashed boundary circle in Figure \ref{fig:extrarelation}, then there are only two regular isotopy classes of embedded arcs in the dashed circle minus $p$ with two fixed boundary points. Therefore, $\ker \iota_{\#}$ is generated by the relations in Figure \ref{fig:extrarelation}. Note that except the difference near $p$, the remaining part of the multicurves are the same. 

\begin{figure}[!ht]
\begin{minipage}{0.5in}\includegraphics[width=\textwidth]{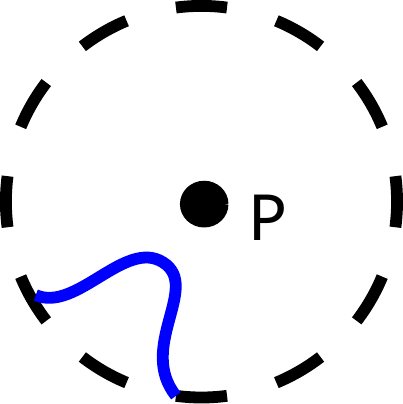}\end{minipage} $\;-\;$ \begin{minipage}{0.5in}\includegraphics[width=\textwidth]{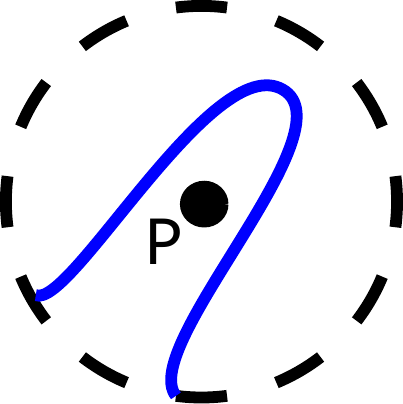}\end{minipage}
\caption{Generators of $\ker \iota_{\#}$}
\label{fig:extrarelation}
\end{figure}

Therefore, to complete the proof, it is sufficient to show that each relation in Figure \ref{fig:extrarelation} is in $K$. This is proved in Proposition \ref{prop:kerisK}.
\end{proof}

\begin{proposition}\label{prop:kerisK}
Every element in Figure \ref{fig:extrarelation} is in $K$.
\end{proposition}

\begin{proof}

In this proof, we mainly use multicurves instead of their regular isotopy classes. Whenever we want to describe their regular isotopy classes, we will explicitly mention it. 

Suppose that there is a multicurve $\alpha$ with a strand sufficiently close to $p$. Let $\alpha^{c}$ be a new multicurve that is obtained by crossing $p$. In other words, $\alpha - \alpha^{c}$ is the relation in Figure \ref{fig:extrarelation}. 

\textsf{Step 1.} Reduction to curves without crossings. 

The Skein relations, the Puncture-Skein relations, and the `crossing $p$' relations are completely local. For any multicurve $\alpha$, by applying Puncture-Skein relations and Skein relations repeatedly, we may obtain an $R_{q, n}$-linear combination $\alpha = \sum c_{I}\alpha_{I}$ of multicurves without any crossings (but the multicurves may have some trivial loops and punctured loops -- so this may not be a linear combination of reduced multicurves). For each $\alpha_{I}$, there is a unique connected component that contains the strand crossing $p$, and $\alpha_{I}^{c}$ is obtained by applying this crossing operation for the connected component. Thus, if we know that $\alpha = \alpha^{c} \in \cA_{q}(\RR^{2}_{n})/K$ for any curves without crossings, then $\sum c_{I}\alpha_{I} = \sum c_{I}\alpha_{I}^{c}$ in $\cA_{q}(\RR^{2}_{n})/K$. After that, we may apply all of the Skein relations and Puncture-Skein relations backward to get $\alpha^{c}$. Thus, $\alpha - \alpha^{c} \in K$. 

\textsf{Step 2.} Reduction to reduced arcs.

Now, suppose that $\alpha$ is a curve without intersection. Then, by using the Puncture-Skein relations only (see Figure \ref{fig:recursion} or \cite[Proposition 2.2]{BKPW16Involve}), we may describe $\alpha$ as a polynomial with respect to (1) reduced arcs, (2) regular isotopy classes of waterdrops $\omega_{i}$ and $\overline{\omega_{i}}$, and (3) the regular isotopy class of the trivial loop. If $\alpha$ is a curve isotopic to the trivial loop, then $\alpha - \alpha^{c} = -(\delta + q^{2} + q^{-2}) \in K$ follows from the Big circle relation. If $\alpha$ is a curve isotopic to a waterdrop $\omega_{i}$ (Example \ref{ex:skeinexample}), then $\alpha - \alpha^{c} = -(\gamma_{ii}^{+} - \omega_{i}) \in K$. The case of $\overline{\omega_{i}}$ is also obtained by Remark \ref{rem:oppositewaterdrop}. Therefore, by a similar argument to that in Step 1, it is now sufficient to prove the statement for reduced arcs.

\textsf{Step 3.} Complexity measure for reduced arcs. 

Let $\alpha$ be a reduced arc connecting two vertices $v_{i}$ and $v_{j}$. From the reducedness (Definition \ref{def:reducedcurves}), the two end vertices are different, so $v_{i} \ne v_{j}$. We define $w(\alpha)$ as the number of connected components of $\alpha \cap \mathrm{int}\; P$. Each connected component $\sigma$ of $\alpha \cap \mathrm{int} \;P$ divides $P$ into two components. Let $e(\sigma)$ be the smaller number of vertices in one of the components, not counting the end vertices of $\sigma$. We define $e(\alpha) = \mathrm{min}\; \{e(\sigma)\}$. If there is no component in $P$, we set $e(\alpha) = 0$. Finally, for each reduced arc $\alpha$, we define its \emph{complexity} as $c(\alpha) = (w(\alpha), e(\alpha)) \in \NN^{2}$ (see Figure \ref{fig:complexityexample} for an example). Give the lexicographical order on $\NN^{2}$. 

\begin{figure}[!ht]
\begin{minipage}{0.9in}\includegraphics[width=\textwidth]{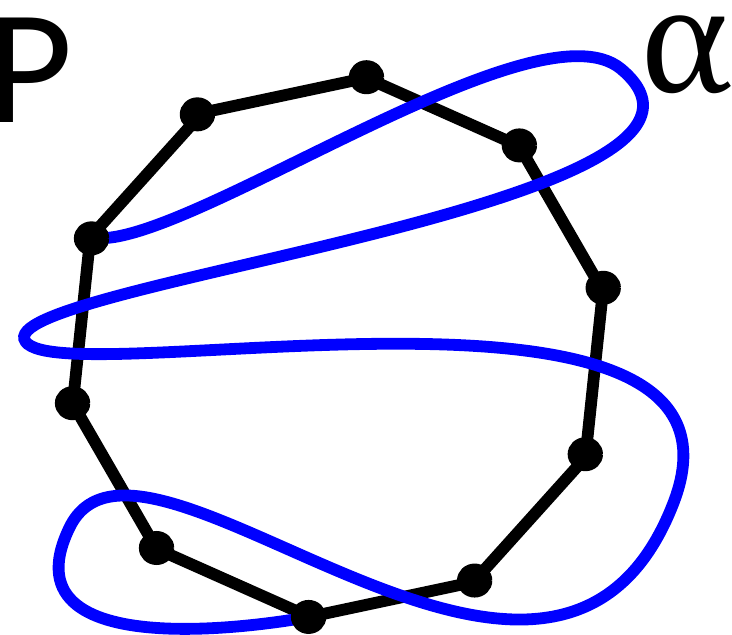}\end{minipage}
\caption{A curve $\alpha$ with $c(\alpha) = (4, 2)$}
\label{fig:complexityexample}
\end{figure}

\textsf{Step 4.} Proof for the reduced arcs. 

We use transfinite induction on the complexity of $\alpha$. If $w(\alpha) = 0$, then $\alpha$ is an arc on the outside of $P$, and is regularly isotopic to $\gamma_{ij}^{+}$ or $\gamma_{ij}^{-}$. Thus, $\alpha - \alpha^{c} = \pm (\gamma_{ij}^{+} - \gamma_{ij}^{-}) \in K$. Note, that in this case, $c(\alpha) = (0, 0)$ as there is no component inside $P$. 

Suppose that the statement is true for all reduced arcs with complexity less than $(r, s)$ and let $\alpha$ be a reduced arc with $c(\alpha) = (w(\alpha), e(\alpha)) = (r, s)$. Then, $r \ge 1$, and there must be a component $\sigma$ of $\alpha \cap \mathrm{int}\; P$ such that $e(\sigma) = e(\alpha)$. If $e(\alpha) = e(\sigma) = 0$, then $\alpha$ has a turn-back near $\sigma$ (Figure \ref{fig:turnback}). By moving $\sigma$ to the outside of $P$, we obtain a new reduced arc $\alpha'$, which is isotopic to $\alpha$ but $c(\alpha') < c(\alpha)$, as $w(\alpha') = w(\alpha) - 1$. Then, by our induction hypothesis, in $\cA_{q}(\RR^{2}_{n})/K$, $\alpha = \alpha' = {\alpha'}^{c} = \alpha^{c}$.

\begin{figure}[!ht]
\begin{minipage}{0.8in}\includegraphics[width=\textwidth]{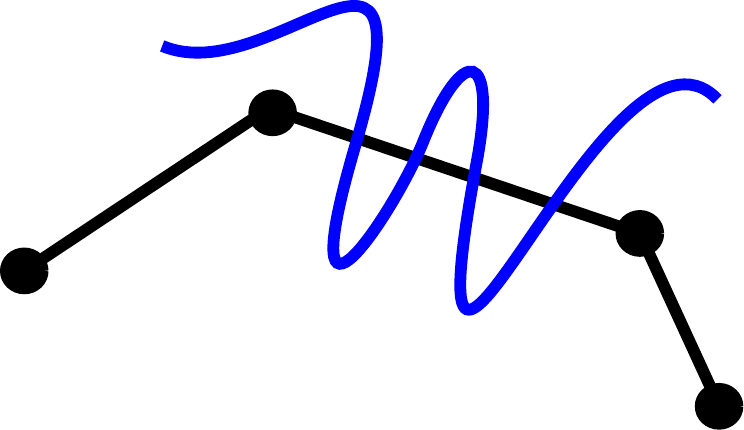}\end{minipage} $\; \Rightarrow \;$ \begin{minipage}{0.8in}\includegraphics[width=\textwidth]{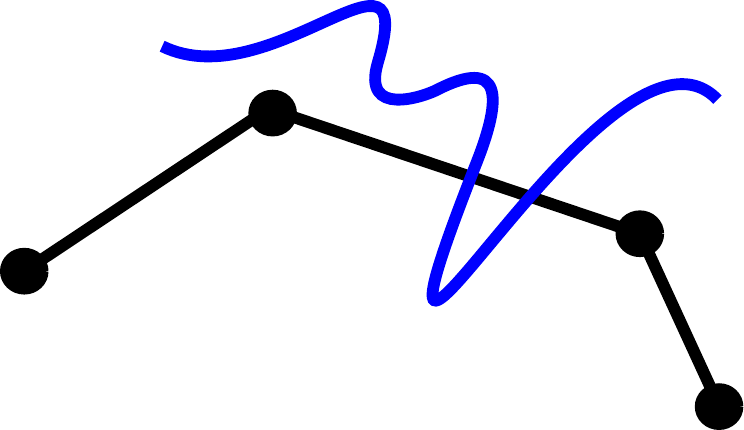}\end{minipage} \quad \quad
\begin{minipage}{0.8in}\includegraphics[width=\textwidth]{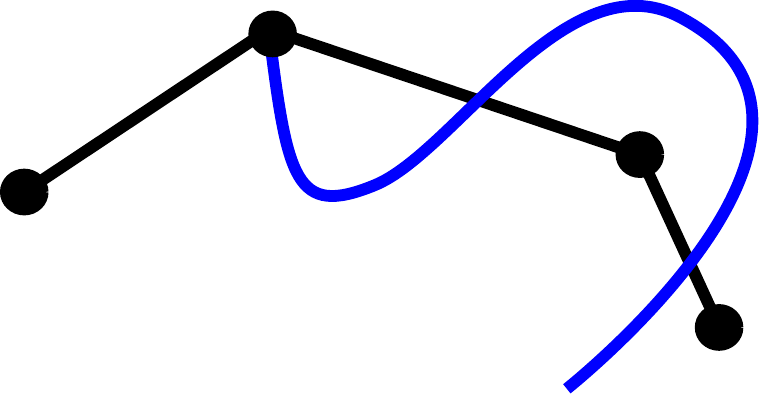}\end{minipage} $\; \Rightarrow \;$ \begin{minipage}{0.8in}\includegraphics[width=\textwidth]{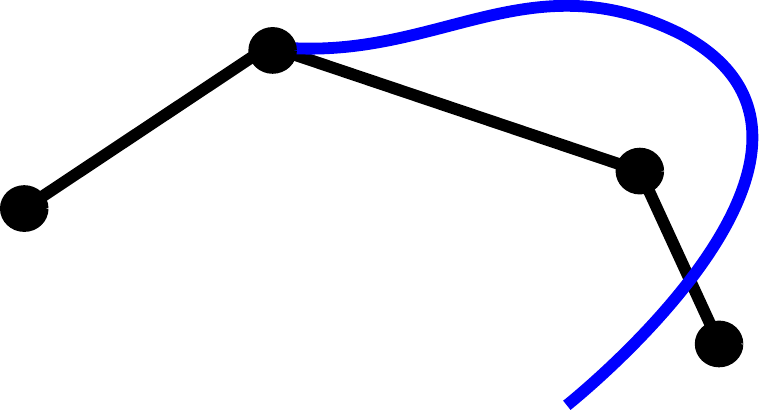}\end{minipage}
\caption{Two types of turn-backs and their removal}
\label{fig:turnback}
\end{figure}

Suppose now that $e(\alpha) > 0$. Pick a component $\sigma$ of $\alpha \cap \mathrm{int}\; P$ such that $e(\sigma) = e(\alpha)$. We deform $\sigma$ toward the region which contains $|e(\sigma)|$ vertices until one end of $\sigma$ hits a vertex (say $v_{i}$). By the Puncture-Skein relation, we obtain $\alpha = q^{-\frac{1}{2}}v_{i}\delta_{1}\delta_{2} - q^{-1}\epsilon$ (Figure \ref{fig:complexitydeduction}). Then, $c(\delta_{1})$, $c(\delta_{2})$, and $c(\epsilon)$ are strictly less than $c(\alpha)$. Without loss of generality, we may assume that $\delta_{2}$ is the curve containing the part that crosses $p$. By our induction hypothesis, in $\cA_{q}(\Sigma_{0, n})/K$, 
\[
	\alpha = q^{-\frac{1}{2}}v_{i}\delta_{1}\delta_{2} - q^{-1}\epsilon = q^{-\frac{1}{2}}v_{i}\delta_{1}\delta_{2}^{c} - q^{-1}\epsilon^{c} = \alpha^{c}.
\]
Therefore, $\alpha - \alpha^{c} \in K$. 
\end{proof}

\begin{figure}[!ht]
\begin{minipage}{0.8in}\includegraphics[width=\textwidth]{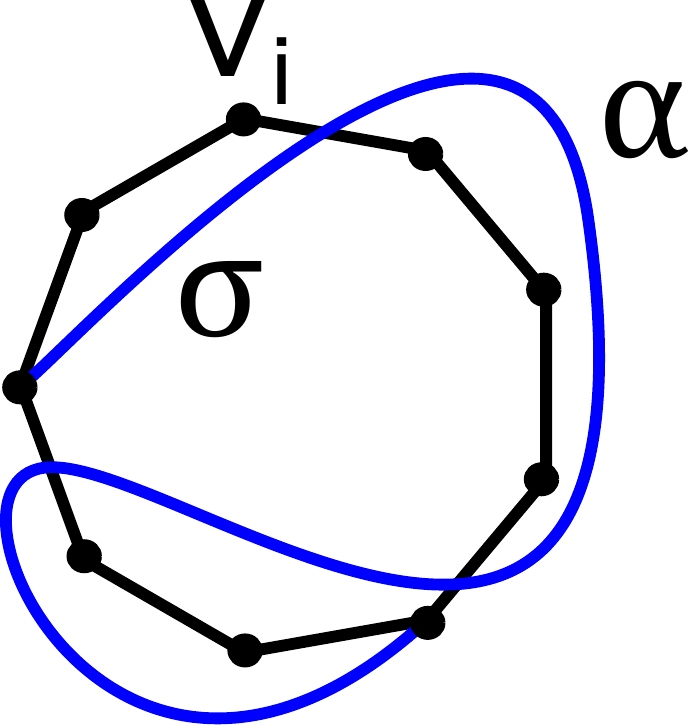}\end{minipage} $\; = \; q^{-\frac{1}{2}}v_{i}\;$ \begin{minipage}{0.8in}\includegraphics[width=\textwidth]{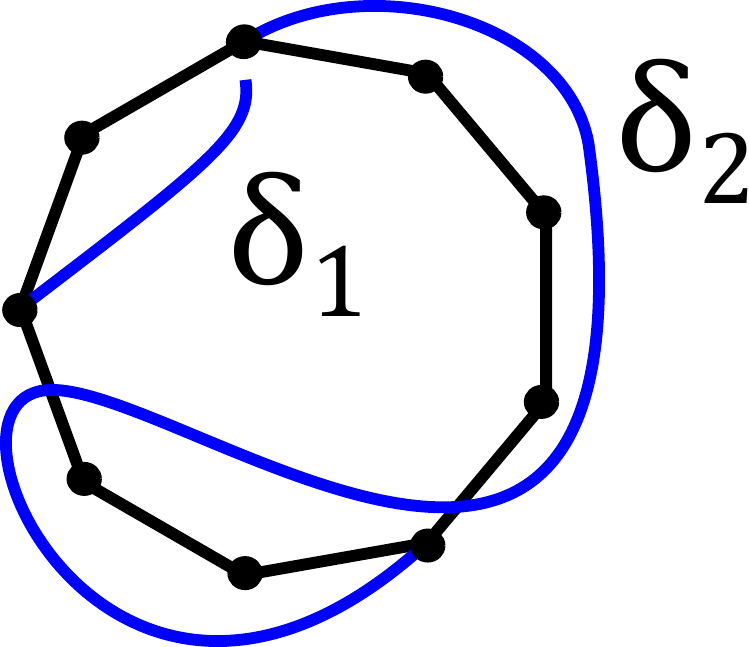}\end{minipage} $\; - \; q^{-1}\;$ \begin{minipage}{0.8in}\includegraphics[width=\textwidth]{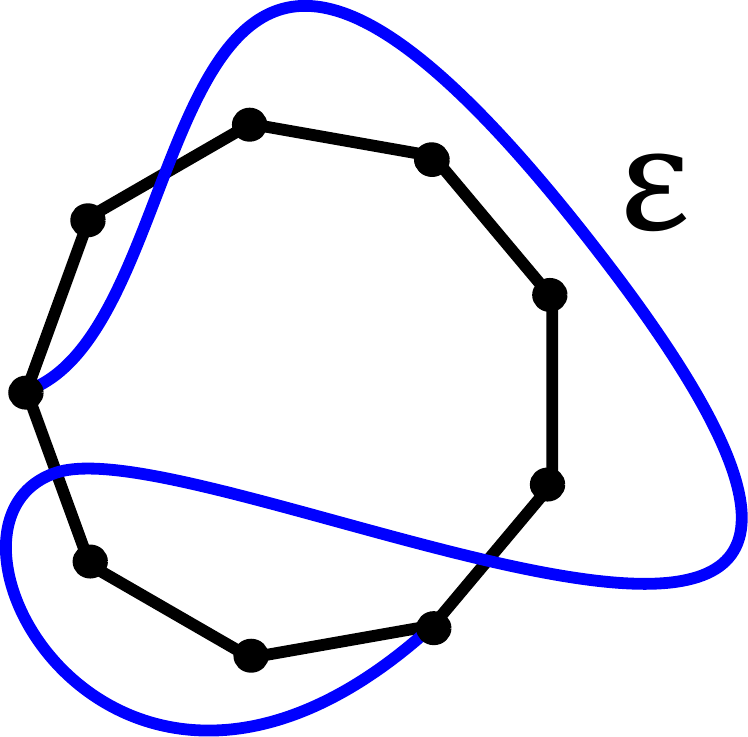}\end{minipage}
\caption{A complexity deduction $\alpha = q^{-\frac{1}{2}}v_{i}\delta_{1}\delta_{2} - q^{-1}\epsilon$}
\label{fig:complexitydeduction}
\end{figure}

\begin{remark}\label{rem:smalln}
It is straightforward to verify the coincidence of our presentation in Theorem \ref{thm:matinthm} with the one in \cite{BKPW16Involve} for $n \le 3$. By a direct computation, we have 
\[
	\cA_{q}(\Sigma_{0, 2}) \cong \ZZ[q^{\pm\frac{1}{2}}, v_{1}^{\pm}, v_{2}^{\pm}, \beta_{12}]/(v_{1}v_{2}\beta_{12}^{2} - 2 + q^{2} + q^{-2}),
\]
\[
	\cA_{q}(\Sigma_{0, 1}) \cong \ZZ[q^{\pm\frac{1}{2}}, v_{1}^{\pm}]/((q+q^{-1}) + (q^{2} + q^{-2})) \cong \ZZ[q^{\pm\frac{1}{2}}, v_{1}^{\pm}]/((q+q^{-1}-1)(q+q^{-1}+2)),
\]
\[
	\cA_{q}(\Sigma_{0, 0}) \cong \ZZ[q^{\pm \frac{1}{2}}].
\]
These presentations show the pathological behavior of $\cA_{q}(\Sigma_{0, n})$ for small $n$. For $n \le 2$, $\cA_{q}(\Sigma_{0, 2})$ is commutative. $\cC(\Sigma_{0, 2}) \cong \ZZ[v_{1}^{\pm}, v_{2}^{\pm}, \beta_{12}]/(\beta_{12}^{2})$ and $\cA_{q}(\Sigma_{0, 1})$ are not integral domains, and $\cA_{q}(\Sigma_{0, 1})$ is not a free $\ZZ[q^{\pm \frac{1}{2}}, v_{1}^{\pm}]$-module anymore. 
\end{remark}

\bibliographystyle{alpha}

\begin{thebibliography}{BKPW16b}

\bibitem[BKPW16a]{BKPW16JKTR}
M. Bobb, S. Kennedy, H. Wong, and D. Peifer.
\newblock Roger and Yang's Kauffman bracket arc algebra is finitely generated. 
\newblock {\em J. Knot Theory Ramifications} 25 (2016), no. 6, 1650034, 14 pp.

\bibitem[BKPW16b]{BKPW16Involve}
M. Bobb, D. Peifer, S. Kennedy, and H. Wong.
\newblock Presentations of Roger and Yang's Kauffman bracket arc algebra. 
\newblock {\em Involve} 9 (2016), no. 4, 689--698.

\bibitem[BW11]{BonahonWong11}
F. Bonahon and H. Wong. 
\newblock Quantum traces for representations of surface groups in $\mathrm{SL}_{2}(\CC)$.
\newblock {\em Geom. Topol.} 15 (2011), no. 3, 1569--1615.

\bibitem[Bul97]{Bullock97}
D. Bullock.
\newblock Rings of ${\rm SL}_2({\bf C})$-characters and the Kauffman bracket skein module.
\newblock {\em Commentarii Mathematici Helvetici} 72 (1997), no. 4, 521--542. 

\bibitem[Bul99]{Bullock99}
D. Bullock.
\newblock A finite set of generators for the Kauffman bracket skein algebra. 
\newblock {\em Math. Z.} 231 (1999), no. 1, 91--101.

\bibitem[BFKB99]{BullockFrohmanJKB99}
D. Bullock, C. Frohman, and J. Kania-Bartoszy\'nska. 
\newblock Understanding the Kauffman bracket skein module. 
\newblock {\em J. Knot Theory Ramifications} 8 (1999), no. 3, 265--277. 

\bibitem[BLF05]{BullockLoFaro05}
D. Bullock, W. Lo Faro.
\newblock The Kauffman bracket skein module of a twist knot exterior. 
\newblock {\em Algebr. Geom. Topol.} 5 (2005), 107--118. 

\bibitem[BP00]{BullockPrzytycki00}
D. Bullock and J. H. Przytycki.
\newblock Multiplicative structure of Kauffman bracket skein module quantizations. 
\newblock {\em Proc. Amer. Math. Soc.} 128 (2000), no. 3, 923--931. 

\bibitem[FST08]{FominShapiroThurston08}
S. Fomin, M. Shapiro, and D. Thurston.
\newblock Cluster algebras and triangulated surfaces. I. Cluster complexes.
\newblock {\em Acta Math.} 201 (2008), no. 1, 83--146.

\bibitem[FZ03]{FominZelevinsky03}
S. Fomin and A. Zelevinsky.
\newblock Cluster algebras II: finite type classification.
\newblock {\em Invent. Math.} 154 (2003) 63--121.

\bibitem[GH94]{GriffithsHarris94}
P. Griffiths and J. Harris.
\newblock {\em Principles of algebraic geometry.} 
Reprint of the 1978 original. Wiley Classics Library. John Wiley \& Sons, Inc., New York, 1994. xiv+813 pp.

\bibitem[Har77]{Hartshorne77}
R. Hartshorne.
\newblock {\em Algebraic geometry.}
\newblock Graduate Texts in Mathematics, No. 52. Springer-Verlag, New York-Heidelberg, 1977. xvi+496 pp.

\bibitem[Kau87]{Kauffman87}
L. H. Kauffman. 
\newblock State models and the Jones polynomial. 
\newblock {\em Topology} 26 (1987), 395--407.

\bibitem[MS15]{MaclaganSturmfels15}
D. Maclagan and B. Sturmfels.
\newblock {\em Introduction to tropical geometry.}
\newblock Graduate Studies in Mathematics, 161. American Mathematical Society, Providence, RI, 2015. xii+363 pp.

\bibitem[MS19]{MoonSwinarski19}
H.-B. Moon and D. Swinarski.
\newblock On the $S_{n}$-invariant F-conjecture. 
\newblock {\em J. Algebra} 517 (2019), 439--456.

\bibitem[MW19]{MoonWong19}
H.-B. Moon and H. Wong. 
\newblock The Roger-Yang skein algebra and the decorated Teichm\"uller space.
\newblock To appear in {\em Quantum Topol.} arXiv:1909.03085. 

\bibitem[MW20]{MoonWong20}
H.-B. Moon and H. Wong. 
\newblock The Roger-Yang skein algebra and cluster algebra. 
\newblock In preperation. 

\bibitem[Mul16]{Muller16}
G. Muller. 
\newblock Skein and cluster algebras of marked surfaces. 
\newblock {\em Quantum Topol.} 7 (2016), no. 3, 435--503.

\bibitem[Pen87]{Penner87}
R. C. Penner.
\newblock The decorated Teichm\"uller space of punctured surfaces. 
\newblock {\em Comm. Math. Phys.} 113 (1987), 299--339.

\bibitem[Prz91]{Przytycki91}
J. H. Przytycki.
\newblock Skein modules of 3-manifolds. 
\newblock {\em Bull. Polish Acad.} Sci 39 (1991) 91--100.

\bibitem[Prz99]{Przytycki99}
J. H. Przytycki.
\newblock Fundamentals of Kauffman bracket skein modules.
\newblock {\em Kobe J. Math.} 16 (1999), no. 1, 45--66. 

\bibitem[PS00]{PrzytyckiSikora00}
J. H. Przytycki and A. Sikora.
\newblock On skein algebras and $\mathrm{SL}_{2}(\CC)$-character varieties.
\newblock {\em Topology} 39 (2000), no. 1, 115--148.

\bibitem[RY14]{RogerYang14}
J. Roger and T. Yang.
\newblock The skein algebra of arcs and links and the decorated Teichm\"uller space.
\newblock {\em J. Differential Geom.} 96 (2014), no. 1, 95--140. 

\bibitem[Sta12]{Stanley12}
R. Stanley. 
\newblock {\em Enumerative combinatorics.} Volume 1. Second edition. 
\newblock Cambridge Studies in Advanced Mathematics, 49. Cambridge University Press, Cambridge, 2012. xiv+626 pp.

\bibitem[Stu08]{Sturmfels08}
B. Sturmfels.
\newblock {\em Algorithms in invariant theory.} Second edition. 
\newblock Texts and Monographs in Symbolic Computation. Springer, Wien-NewYork-Vienna, 2008. vi+197 pp.

\bibitem[Tur88]{Turaev88}
V. G. Turaev.
\newblock The Conway and Kauffman modules of a solid torus.
\newblock {\em Zap. Nauchn. Sem. Leningrad. Otdel. Mat. Inst. Steklov.} (LOMI) 167 (1988), Issled. Topol. 6, 79--89, 190.

\end{thebibliography}

\end{document}